\documentclass[11pt]{article}

\usepackage[margin=1.25in]{geometry}

\usepackage{amssymb,amsmath}
\usepackage{amsthm}
\usepackage{hyperref}
\usepackage{mathrsfs}
\usepackage{textcomp}
\usepackage{enumerate}
\usepackage{framed}
\hypersetup{
    colorlinks,%
    citecolor=blue,%
    filecolor=blue,%
    linkcolor=blue,%
    urlcolor=black
}

\newtheorem{theorem}{Theorem}[section]
\newtheorem{definition}[theorem]{Definition}
\newtheorem{lemma}[theorem]{Lemma}

\newtheorem{proposition}[theorem]{Proposition}
\newtheorem{corollary}[theorem]{Corollary}

\makeatletter

\newdimen\bibspace
\renewenvironment{thebibliography}[1]{%
 \section*{\refname 
       \@mkboth{\MakeUppercase\refname}{\MakeUppercase\refname}}%
     \list{\@biblabel{\@arabic\c@enumiv}}%
          {\settowidth\labelwidth{\@biblabel{#1}}%
           \leftmargin\labelwidth
           \advance\leftmargin\labelsep
           \itemsep\bibspace
           \parsep\z@skip     %
           \@openbib@code
           \usecounter{enumiv}%
           \let\p@enumiv\@empty
           \renewcommand\theenumiv{\@arabic\c@enumiv}}%
     \sloppy\clubpenalty4000\widowpenalty4000%
     \sfcode`\.\@m}
    {\def\@noitemerr
      {\@latex@warning{Empty `thebibliography' environment}}%
     \endlist}

\makeatother

           \newcommand{\ud}{\mathrm{d}}
\newcommand{\be}{\begin{equation}}      \newcommand{\ee}{\end{equation}}

\begin{document}

\title{Compactness of solutions to higher order elliptic equations}

\author{{\medskip Miaomiao Niu\thanks{M. Niu is supported by National Science Foundation of China (11901020), Beijing National Science Foundation (1204026) and Science and technology Project of Beijing Municipal Commission of Education China (KM202010005027).},\ \  Zhongwei Tang\thanks{Z. Tang is supported by National Science Foundation of China (12071036).},\ \  Ning Zhou} }

\date{}

\maketitle

\begin{abstract}
We use blow up analysis for local integral equations to prove compactness of solutions to higher order critical elliptic equations provided the potentials only have non-degenerate zeros. Secondly, corresponding to Schoen's Weyl tensor vanishing conjecture for the Yamabe equation on manifolds, we establish a Laplacian vanishing rate of the potentials at blow up points of solutions.
\end{abstract}
{\bf Key words:} Higher order equations, Compactness, Blow up analysis.

{\noindent\bf Mathematics Subject Classification (2010)}\quad 35B44 · 35G20 · 35J61

\section{Introduction}

This paper aims to study the compactness of solutions for the higher order critical elliptic equations
\be\label{eq:sigmaSchrodinger}
(-\Delta)^{\sigma} u-a(x)u=u^{\frac{n+2 \sigma}{n-2 \sigma}} \quad \text { in }\, B_{3},
\ee
where $1 \leq \sigma<n/2$ is an integer, the potential $a(x)$ is nonnegative and smooth in $B_3$, $n \geq 4\sigma$. We also prove that the Laplacian of $a(x)$ has to vanish at possible blow up points of a sequence of blowing up solutions. 

When $\sigma=1$ and the equation is defined in $\mathbb{R}^n$, \eqref{eq:sigmaSchrodinger} becomes the semilinear Schr\"odinger equation. If $a(x)\equiv \lambda$, where $\lambda$ is a positive constant, \eqref{eq:sigmaSchrodinger} equipped with the Dirichlet boundary conditions is called the Br\'ezis-Nirenberg problem. Compactness of finite energy changing-signs solutions of Br\'ezis-Nirenberg problem was established in dimensions $n>6$ by Devillanova-Solimini \cite{DSConcentration2002}. For the fractional case $0<\sigma<1$, compactness of solutions of \eqref{eq:sigmaSchrodinger} has been studied by Niu-Peng-Xiong \cite{NPXCompactness2018} and compactness of solutions of fractional Br\'ezis-Nirenberg problem has been studied by Yan-Yang-Yu \cite{YYYEquations2015}.

If $a(x)\equiv0$ and the equation is defined in $\mathbb{R}^n$, \eqref{eq:sigmaSchrodinger} becomes
\be\label{eq:a==0}
(-\Delta)^{\sigma} u=u^{\frac{n+2 \sigma}{n-2 \sigma}} \quad \text { in }\, \mathbb{R}^n,
\ee
the compactness of solutions of \eqref{eq:a==0} fails for $0<\sigma<n/2$. Indeed, by the classification theorem of the solutions of \eqref{eq:a==0} (see, e.g. \cite{CGSAsymptotic1989}, \cite{LA1998}, \cite{WXClassification1999}, \cite{CLOClassification2006}, \cite{JLXOnI2014}), the positive solutions of \eqref{eq:a==0} are the form
$$
u_{x_0, \lambda}(x)=c(n,\sigma)\Big(\frac{\lambda}{1+\lambda^2|x-x_0|^2}\Big)^{\frac{n-2\sigma}{2}},
$$
where $c(n,\sigma)$ is a constant depending of $n$ and $\sigma$, $\lambda>0$ and $x_0\in \mathbb{R}^n$. Clearly, $u_{x_0, \lambda}(x_0)\to +\infty$ as $\lambda\to +\infty$.

The compactness of solutions to the Nirenberg problems has been studied extensively. In this case, the right-hand side of Eq. \eqref{eq:a==0} becomes $K(x)u^{\frac{n+2\sigma}{n-2\sigma}}$. Let $K(x)$ satisfies some flatness assumptions. For $\sigma=1$, $n=3$, the compactness of solutions to the Nirenberg problems was proved in Chang-Gursky-Yang \cite{CGYThe1993} and Schoen-Zhang \cite{SZPrescribed1996}. For $\sigma=1$, $n \geq 4$, it was proved by Li \cite{LPrescribing1995}. For $\sigma=2$ and $n\geq5$, it was proved by Djadli-Malchiodi-Ahmedou \cite{DMAPrescribingI2002, DMAPrescribingII2002} and Felli \cite{FExistence2002}. For $\sigma \in(0,1)$, it was proved by Jin-Li-Xiong \cite{JLXOnI2014, JLXOnII2015}. For $\sigma \in(0,n/2)$, it was proved by Jin-Li-Xiong \cite{JLXThe2017}.

The first main result of the paper is as follows.

\begin{theorem}\label{thm:1.1}
Suppose that $1 \leq \sigma<n / 2$ and $\sigma$ is an integer. Let $u \in C^{2 \sigma}(B_{3})$ be a nonnegative solution of
\be\label{eq:Differential}
(-\Delta)^{\sigma} u-a(x)u=u^{\frac{n+2 \sigma}{n-2 \sigma}} \quad \text { in }\, B_{3},
\ee
where $a(x)$ is a nonegative smooth function in $B_3$ and $n\geq 4\sigma$. Suppose
\be\label{eq:sign conditions}
(-\Delta)^{k} u \geq 0 \quad \text { in }\, B_3, \quad k=1, \cdots, \sigma-1.
\ee
If either
\begin{enumerate}[(i)]
  \item $a>0$ in $B_{2}$, or
  \item $\Delta a>0$ on $\{x: a(x)=0\} \cap B_{2}$ and $n \geq 4 \sigma+2$
\end{enumerate}
holds, then
$$
\|u\|_{C^{2\sigma}(B_{1})} \leq C,
$$
where $C>0$ depends only on $n, \sigma,\|a\|_{C^{4}(B_{3})}$ and $\inf _{B_{2}} a$ if $(i)$ holds, otherwise it depends only on $n, \sigma,\|a\|_{C^{4}(B_{3})}$ and $\inf _{\{x: a(x)=0\} \cap B_{2}} \Delta a .$
\end{theorem}

By Green's representation, we can write the differential equations \eqref{eq:Differential} into the integral equation \eqref{eq:Integral} below, and we will work in the integral equation setting. This is inspired by a unified approach to the Nirenberg problem and its generalizations studied by Jin-Li-Xiong in \cite{JLXThe2017}, see also \cite{LXCompactness2019} and \cite{JXAsymptotic2021}.

Assume that the dimension $n \geq 1$ and $1\leq\sigma<n / 2$ is a real number. We consider the local integral equation involving the Riesz potential

\begin{equation}\label{eq:Integral}
u(x)=\int_{B_3} \frac{a(y) u(y)+u(y)^{\frac{n+2 \sigma}{n-2 \sigma}}}{|x-y|^{n-2\sigma}}\,\ud y+h(x),\,\quad u \geq 0 \quad \text { in }\, B_{3},
\end{equation}
where $0\leq h\in C^{\infty}(B_3)$ satisfy
\be\label{eq:hsatisfy}
\max _{\overline{B}_{2}} h \leq A_2 \min _{\overline{B}_{2}} h \, \text{ and }\, \sum_{j=1}^{3}r^j|\nabla^j h(x)|\leq A_2\|h\|_{L^{\infty}(B_r(x))}
\ee
for all $x\in B_2$, $0<r<1$, and some positive constant $A_2$. Under the assumptions of Theorem \ref{thm:1.1}, one can rewrite the equation \eqref{eq:Differential} as \eqref{eq:Integral} (after some scaling, see details in Section 2). Note that if $u\in C^{2\sigma}(B_3)$ is a solution of \eqref{eq:Integral}, then one can verify that the function $\tilde{h}:=\int_{B_3\backslash B_1} (a(y) u(y)+u(y)^{\frac{n+2 \sigma}{n-2 \sigma}})|x-y|^{2\sigma-n}\,\ud y$ is smooth and nonegative in $B_1$. Thus, one always can consider the equation \eqref{eq:Integral} in a smaller ball with the same assumptions. For simplicity, in the whole paper, we use $C^{\alpha}(B_3)$ to denote $C^{[\alpha], \alpha-[\alpha]}(B_3)$.

Next we state the corresponding result for the integral equation \eqref{eq:Integral}.

\begin{theorem}\label{thm:1.2}
Let $u \in C^{2\sigma}(B_{3})$ be a solution of \eqref{eq:Integral} with $a \geq 0$ and $n \geq 4 \sigma .$ If either
\begin{enumerate}[(i)]
  \item $a>0$ in $B_{2}$, or
  \item $\Delta a>0$ on $\{x: a(x)=0\} \cap B_{2}$ and $n \geq 4 \sigma+2$
\end{enumerate}
holds, then
$$
\|u\|_{C^{2\sigma}(B_{1})} \leq C,
$$
where $C>0$ depends only on $n, \sigma,\|a\|_{C^{4}(B_{3})}$ and $\inf _{B_{2}} a$ if $(i)$ holds, otherwise it depends only on $n, \sigma,\|a\|_{C^{4}(B_{3})}$ and $\inf _{\{x: a(x)=0\} \cap B_{2}} \Delta a .$
\end{theorem}

In view of Proposition \ref{pro:Dif-Int}, Theorem \ref{thm:1.1} follows from Theorem \ref{thm:1.2}.

The Yamabe equation on Riemannian manifolds have been studied very extensively, whose potential is the scalar curvature multiplied by a constant. Compactness and blow-up phenomenon of solutions to the Yamabe equation have been well understood. A conjecture due to Schoen asserts that the Weyl tensor should vanish up to $[\frac{n-6}{2}]$-th order derivatives at a blow up point, where $n$ is the dimension of manifolds. This conjecture has been verified in dimensions 6 and 7 by Marques \cite{MA2005} and, independently, in dimensions 6-11 by Li-Zhang \cite{LZCompactnessII2005,LZCompactnessIII2007}. M. Khuri, F. Marques, and R. Schoen \cite{KMSA2009} proved this conjecture up to dimension 24. If $n \geq 25$, a counterexample was obtained by Marques \cite{MBlow2009}. Consequently, the set of solutions of the Yamabe equation is compact in $C^{2}$ if the Weyl tensor or some of its derivatives of order $\leq[\frac{n-6}{2}]$ does not vanish everywhere in dimension less than 24. If the Weyl tensor does not vanish everywhere, compactness was proved in all dimensions $n \geq 6$ by \cite{LZCompactnessII2005,MA2005}. Similar phenomenon has been proved recently by Li-Xiong \cite{LXCompactness2019} for the fourth order $Q$-curvature equation in dimension $n \geq 8$.

Another purpose of paper is to establish an analogue for Yamabe type equations with non-geometric potentials. The second main result in this paper is

\begin{theorem}\label{thm:1.3}
Suppose that $1 \leq \sigma<n / 2$. Let $u_{i} \in C^{2}(B_{3})$, $i=1,2, \cdots$, be a nonegative solution of
\begin{equation}\label{5}
u_i(x)=\int_{B_3} \frac{a_i(y) u_i(y)+u_i(y)^{\frac{n+2 \sigma}{n-2 \sigma}}}{|x-y|^{n-2\sigma}}\,\ud y+h_i(x) \quad \text { in }\, B_{3},
\end{equation}
where $a_{i} \geq 0$, $\|a_{i}\|_{C^{4}(B_{3})} \leq A_{0}$ for some $A_{0}>0$ and $a_{i} \rightarrow a$ in $C^{4}(B_{3})$. Suppose that $\Delta a_{i} \geq 0$ in $\{x: a_{i}(x)<\varepsilon\} \cap B_{2}$ for some $\varepsilon>0$ independent of $i$ and $n \geq 4 \sigma+2 .$ If $x_{i} \rightarrow \bar{x} \in B_{1}$ and $u_{i}(x_{i}) \rightarrow \infty$ as $i \rightarrow \infty$, then $a(\bar{x})=\Delta a(\bar{x})=0 .$ Furthermore,
\begin{enumerate}[(i)]
  \item If $4 \sigma+2 \leq n<4 \sigma+4$, we can find $x_{i}^{\prime} \rightarrow \bar{x}$ such that
  $$
a_{i}(x_{i}^{\prime}) u(x_{i}^{\prime})^{\frac{2}{\sigma+1}} (\ln u_{i}(x_{i}^{\prime}))^{-1}+\Delta a_{i}(x_{i}^{\prime}) \leq C(\ln u_{i}(x_{i}^{\prime}))^{-1}
$$
  for $n=4 \sigma+2$ and
  $$
a_{i}(x_{i}^{\prime}) u(x_{i}^{\prime})^{\frac{4}{n-2 \sigma}}+\Delta a_{i}(x_{i}^{\prime}) \leq C u(x_{i}^{\prime})^{\frac{2(4 \sigma+2-n)}{n-2 \sigma}}
$$
  for $4 \sigma+2<n<4 \sigma+4$, where $C>0$ depends only on $n, \sigma, \varepsilon$ and $A_{0}$.
  \item If $n \geq 4 \sigma+4$, assume that
  \begin{equation}\label{6}
x_{i} \text { is a local maximum point of } u_{i}, \quad \max _{B_{\bar{d}}(x_{i})} u_{i}(x) \leq \bar{b} u_{i}(x_{i})
\end{equation}
  for some positive constants $\bar{b}$ and $\bar{d}$. Then
  $$
a_{i}(x_{i}) u(x_{i})^{\frac{4}{n-2 \sigma}}+\Delta a_{i}(x_{i}) \leq C\left\{\begin{array}{ll}
u_{i}(x_{i})^{-\frac{4 }{n-2 \sigma}} \ln u_{i}(x_{i}) & \text { for }\, n=4 \sigma+4, \\
u_{i}(x_{i})^{-\frac{4 }{n-2 \sigma}} & \text { for }\, n>4 \sigma+4,
\end{array}\right.
$$
  where $C>0$ depends only on $n, \sigma, \varepsilon, A_{0}$, as well as constants $\bar{b}$ and $\bar{d}$.
\end{enumerate}
\end{theorem}

The proofs of Theorem \ref{thm:1.2} and \ref{thm:1.3} make use of important ideas for the proof of compactness of positive solutions of the Yamabe equation, which were outlined first by Schoen \cite{SVariational1989}, as well as methods developed through the work Li \cite{LPrescribing1995}, Li-Zhu \cite{LZYamabe1999}, Li-Zhang \cite{LZA2004, LZCompactnessII2005, LZCompactnessIII2007} and Marques \cite{MA2005}. Since potentials in \eqref{eq:Differential} and \eqref{eq:Integral} are not geometric and their Taylor expansion polynomials of order $\geq 2$ have not to be orthogonal to the zeroth and first order polynomials, it is not possible to construct correctors. It is unclear to us how to show higher order derivatives vanishing estimates.

The rest of this paper is organized as follows. In Section 2, we show the integral representation for nonnegative solutions to the differential equation \eqref{eq:Differential}. In section 3, we establish basic results of so-called isolated simple blow up points. In section 4, we establish the refined quantitative asymptotic analysis mentioned above. In section 5, we prove that the isolated blow up points must to be isolated simple blow up points. The main theorems are proved in section 6.

\section{An integral representation for the integer order cases}
For $n \geq 3$, let $G_{1}(x, y)$ be the Green function of $-\Delta$ on $B_{3}$. Namely, for any $u \in C^{2}(B_{3}) \cap C(\overline{B}_{3})$,
$$
u(x)=\int_{B_{3}} G_{1}(x, y)(-\Delta) u(y) \,\ud y+\int_{\partial B_{3}} H_{1}(x, y) u(y) \,\ud S_{y},
$$
where
$$
H_{1}(x, y)=-\frac{\partial}{\partial \nu_{y}} G_{1}(x, y) \quad \text { for }\, x \in B_{3},\, y \in \partial B_{3}.
$$
By induction, we have, for $2 m<n$ and $u \in C^{2 m}(B_{3}) \cap C^{2 m-2}(\overline{B}_{3})$,
\be\label{eq:green}
u(x)=\int_{B_{3}} G_{m}(x, y)(-\Delta)^{m} u(y) \,\ud y+\sum_{i=1}^{m} \int_{\partial B_{3}} H_{i}(x, y)(-\Delta)^{i-1} u(y) \,\ud  S_{y},
\ee
where
$$
G_{m}(x, y)=\int_{B_{3} \times \cdots \times B_{3}} G_{1}(x, y_{1}) G_{1}(y_{1}, y_{2}) \cdots G_{1}(y_{m-1}, y) \,\ud y_{1} \cdots \,\ud y_{m-1},
$$
and
$$
H_{i}(x, y)=\int_{B_{3} \times \cdots \times B_{3}} G_{1}(x, y_{1}) G_{1}(y_{1}, y_{2}) \cdots G_{1}(y_{i-2}, y_{i-1}) H_{1}(y_{i-1}, y) \,\ud  y_{1} \cdots \,\ud y_{i-1}
$$
for $2 \leq i \leq m$.

By direct computations, we have
\begin{equation}\label{eq:Gm}
G_{m}(x, y)=c(n, m)|x-y|^{2 m-n}+A_{m}(x, y),
\end{equation}
where $c(n, m)=\frac{\Gamma((n-2 m)/{2})}{2^{2 m} \pi^{n / 2} \Gamma(m)}$, $A_{m}(\cdot, \cdot)$ is smooth in $B_{3} \times B_{3}$, and
$$
H_{i}(x, y) \geq 0, \quad i=1, \cdots, m.
$$

\begin{proposition}\label{pro:Dif-Int}
Assume as in Theorem \ref{thm:1.1}. Then there exists $\tau>0$ such that
\begin{equation}
u(x)=c(n,m) \int_{B_{\tau}} \frac{a(y)u(y)+u(y)^{\frac{n+2 m}{n-2 m}}}{|x-y|^{n-2 m}} \,\ud y+h_{1}(x) \quad \text { for }\, x \in B_{\tau},
\end{equation}
where $h_{1}(x)$ is a positive smooth function in $B_{\tau}$.
\end{proposition}

\begin{proof}
Without loss of generality, we can assume that $u \in C^{2 m}(\overline{B}_{3})$ and $u>0$ in $\overline{B}_{3}$. Otherwise, we just consider the equation in a smaller ball.

Since $-\Delta u \geq 0$ in ${B}_{3}$ and $u>0$ in $\overline{B}_3$, it follows from the maximum principle that
$$
u(x) \geq c_{1}:=\inf _{\partial B_{3}} u>0 \quad \text { for all }\, x \in \overline{B}_{3}.
$$
Since $a, u \in C^2(B_3)$, we have $au, u^{\frac{n+2 m}{n-2 m}} \in L^{1}(B_{3})$. Thus, one can find $0<\tau<1/4$ such that
$$
\int_{B_{\tau}}|A_{m}(x, y)| (a(y)u(y)+u(y)^{\frac{n+2 m}{n-2 m}}) \,\ud y<\frac{c_{1}}{2} \quad \text { for }\, x \in B_{\tau},
$$
where $A_{m}(x, y)$ is as in \eqref{eq:Gm}. By \eqref{eq:green}, we write
$$
u(x)=c(n,m) \int_{B_{\tau}} \frac{a(y)u(y)+u(y)^{\frac{n+2 m}{n-2 m}}}{|x-y|^{n-2 m}} \,\ud y+h_{1}(x),
$$
where
$$
\begin{aligned}
h_{1}(x)=&\int_{B_{\tau}} A_{m}(x, y)  (a(y)u(y)+u(y)^{\frac{n+2m}{n-2m}}) \,\ud y+\int_{B_{3} \backslash B_{\tau}} G_{m}(x, y) (a(y)u(y)+u(y)^{\frac{n+2m}{n-2m}}) \,\ud  y \\
&+\sum_{i=1}^{m} \int_{\partial B_{3}} H_{i}(x, y)(-\Delta)^{i-1} u(y) \,\ud  S_{y} \\
\geq &-\frac{c_{1}}{2}+\int_{\partial B_{3}} H_{1}(x, y) u(y) \,\ud  S_{y} \\
\geq &-\frac{c_{1}}{2}+\inf _{\partial B_{3}} u\\
=&\frac{c_{1}}{2} \quad \text { for }\, x \in B_{\tau},
\end{aligned}
$$
where we used the sign conditions \eqref{eq:sign conditions} in the first inequality and the fact that $\int_{\partial B_3} H_1(x,y)\,\ud S_{y}=1$ in the second inequality. This is the only place that we used these sign conditions \eqref{eq:sign conditions}. On the other hand, it is easy to check that $h_{1}$ is smooth in $B_{\tau}$. This completes the proof.
\end{proof}

\section{Analysis of isolated blow up points}
We first recall some local estimates for the solutions of the linear integral equation
\begin{equation}\label{eq:LIE}
u(x)=\int_{B_3} \frac{V(y) u(y)}{|x-y|^{n-2 \sigma}}\, \ud y+h(x) \quad \text { in }\, B_{3},
\end{equation}
where $1\leq\sigma<n / 2$, $0\leq h\in C^{\infty}(B_3)$ satisfy \eqref{eq:hsatisfy}. Their proofs can be found in Jin-Li-Xiong \cite{JLXThe2017}, see also Li-Xiong \cite{LXCompactness2019}.

\begin{proposition}\label{pro:JLXHarnack}
Let $0 \leq V \in L^{\infty}(B_{3})$, and let $0 \leq u \in C^0(B_{3})$ satisfy \eqref{eq:LIE}, then we have
$$
\max _{\overline{B}_{2}} u \leq C \min _{\overline{B}_{2}} u,
$$
where $C>0$ depends only on $n, \sigma, A_2,$ and $\|V\|_{L^{\infty}(B_{3})}$.
\end{proposition}

\begin{proposition}\label{pro:JLXSchauder}
Suppose the hypotheses in Proposition \ref{pro:JLXHarnack}. Then $u \in C^{2}(B_{2})$ and
$$
\|u\|_{C^{2}(B_{2})} \leq C\|u\|_{L^{\infty}(B_{3})},
$$
where $C>0$ depends only on $n, \sigma, A_2,$ and $\|V\|_{L^{\infty}(B_{3})}$.

If $V \in C^{1}(B_3)$, then $u \in C^{3}(B_2)$ and
$$
\|\nabla^{3} u\|_{L^{\infty}(B_1)} \leq C\|u\|_{L^{\infty}(B_3)},
$$
where $C>0$ depends only on $n, \sigma, A_2,$ and $\|V\|_{C^{1}(B_{3})}$.
\end{proposition}

Let $\tau_{i}\geq 0$ satisfy $\lim _{i \rightarrow \infty} \tau_{i}=0$, $p_{i}=(n+2 \sigma) /(n-2 \sigma)-\tau_{i}$, $a_{i} \geq 0$ be a sequence
of functions converging to $a$ in $C^{2}(B_{3})$, and $h_i\geq 0$ be a sequence of functions in $C^{\infty}(B_{3})$ satisfying for $0<r<1 / 2$,
\begin{equation}\label{eq:Harnackhi}
\max _{\bar{B}_{r}(x)} h_{i} \leq A_2 \min _{\bar{B}_{r}(x)} h_{i}\, \text{ and }\, \sum_{j=1}^{3} r^{j}|\nabla^{j} h_{i}(x)| \leq A_2\|h_{i}\|_{L^{\infty}(B_{r}(x))}, \quad x \in B_{2}.
\end{equation}
Let $u_{i}\geq 0$ be a sequence of $C^{2}(B_{3})$ solutions of
\begin{equation}\label{eq:intui=}
u_i(x)=\int_{B_3} \frac{a_i(y) u_i(y)+u_i(y)^{p_i}}{|x-y|^{n-2\sigma}}\,\ud y+h_i(x) \quad \text { in }\, B_{3}.
\end{equation}
We say a point $\bar{x} \in B_{2}$ is a blow up point of $\{u_{i}\}$ if $u_{i}(x_{i}) \rightarrow \infty$ for some $x_{i} \rightarrow \bar{x}$.

\begin{definition}\label{def:isolated}
Suppose that $\{u_{i}\}$ satisfies \eqref{eq:intui=}. We say a point $\bar{x} \in B_{2}$ is an isolated blow up point of $\{u_{i}\}$ if there exist $0<\bar{r}<\operatorname{dist}(\bar{x}, \partial B_{3})$, $\bar{C}>0$, and a sequence $x_{i}$ tending to $\bar{x}$, such that $x_{i}$ is a local maximum of $u_{i}$, $u_{i}(x_{i}) \rightarrow \infty$ and
$$
u_{i}(x) \leq \bar{C}|x-x_{i}|^{-2 \sigma /(p_{i}-1)} \quad \text { for all }\, x \in B_{\bar{r}}(x_{i}).
$$
\end{definition}

Let $x_{i} \rightarrow \bar{x}$ be an isolated blow up point of $u_{i}$, define for $0<r<\bar{r}$,
\begin{equation}\label{22}
\bar{u}_{i}(r)=\frac{1}{|\partial B_{r}(x_{i})|} \int_{\partial B_{r}(x_{i})} u_{i} \quad \text { and } \quad \bar{w}_{i}(r)=r^{2 \sigma /(p_{i}-1)} \bar{u}_{i}(r).
\end{equation}

\begin{definition}\label{def:isolated simple}
We say $x_{i} \rightarrow \bar{x} \in B_{2}$ is an isolated simple blow up point if $x_{i} \rightarrow \bar{x}$ is an isolated blow up point such that for some $\rho>0$ (independent of $i$), $\bar{w}_{i}$ has precisely one critical point in $(0, \rho)$ for large $i$.
\end{definition}

In the above, we use $B_{2}$ and $B_{3}$ for conveniences. One can replace them by open sets.

If $x_{i} \rightarrow 0$ is an isolated blow up point, then we will have the following Harnack inequality in the annulus centered at $0$.

\begin{lemma}\label{lem:Harnack}
Suppose that $\{u_{i}\}$ satisfies \eqref{eq:intui=}, and $x_{i} \rightarrow 0$ is an isolated blow up point of $\{u_{i}\}$, i.e., for some positive constants $A_{1}$ and $\bar{r}$ independent of $i$,
\begin{equation}\label{eq:Harnack1}
|x-x_{i}|^{2 \sigma /(p_{i}-1)} u_{i}(x) \leq A_{1} \quad \text { for all }\, x \in B_{\bar{r}}(x_{i}) \subset B_{3}.
\end{equation}
Then for any $0<r<\bar{r}/3$, we have the following Harnack inequality:
$$
\sup_{{B}_{2 r}(x_{i}) \backslash \overline{{B}_{r / 2}(x_{i})}} u_{i}\leq C\inf_{{B}_{2 r}(x_{i}) \backslash \overline{{B}_{r / 2}(x_{i})}} u_{i},
$$
where $C>0$  depends only on $n, \sigma, A_{1}, A_2, \bar{r},$ and $\sup _{i}\|a_{i}\|_{L^{\infty}(B_{\bar{r}}(x_{i}))}$.
\end{lemma}

\begin{proof}
For $0<r<{\bar{r}}/{3}$, define
$$
w_{i}(x):=r^{2 \sigma /(p_{i}-1)} u_{i}(r x+x_{i})
$$
and
$$
\tilde{h}_i(x):=r^{2 \sigma /(p_{i}-1)} h_{i}(r x+x_{i}).
$$
By the equation of $u_i$, we have
$$
w_{i}(x)=\int_{B_{3/r}} \frac{r^{2\sigma}a_{i}(r {y}+x_{i}) w_{i}({y})+w_{i}({y})^{p_{i}}}{|x-y|^{n-2 \sigma}} \,\ud y+\tilde{h}_i(x) \quad \text{ in }\, B_{3}.
$$
Since $x_{i} \rightarrow 0$ is an isolated blow up point of $u_{i}$,
$$
w_{i}(x) \leq A_{1}|x|^{-2 \sigma /(p_{i}-1)} \quad \text { for all }\, x \in B_{3}.
$$
Set $\Omega_{1}=B_{5 / 2} \backslash B_{1 / 4}$, $\Omega_{2}=B_{2} \backslash B_{1 / 2}$ and $V_{i}(y)=r^{2\sigma}a_i(ry+x_i)+w_{i}(y)^{p_{i}-1} .$ Thus $w_i$ satisfies the linear equation
$$
w_{i}(x)=\int_{\Omega_1} \frac{V_i(y)w_i(y)}{|x-y|^{n-2 \sigma}} \,\ud y+\bar{h}_i(x) \quad \text{ in }\, \Omega_1,
$$
where
$$
\bar{h}_{i}(x)=\tilde{h}_{i}(x)+\int_{B_{3 / r} \backslash \Omega_{1}} \frac{V_i(y)w_i(y)}{|x-y|^{n-2 \sigma}} \,\ud y.
$$
It is easily seen that $\|V_{i}\|_{L^{\infty}(\Omega_1)} \leq C.$ Since $a_{i}$ and $w_{i}$ are nonnegative, it follows from \eqref{eq:Harnackhi} that $\max _{\bar{\Omega}_{2}} \bar{h}_{i} \leq C \min _{\bar{\Omega}_{2}} \bar{h}_i$. Covering $\Omega_1$ by finite many balls and applying Proposition \ref{pro:JLXHarnack} to $w_{i}$ yields $\max _{\bar{\Omega}_{2}} w_{i} \leq C \min _{\bar{\Omega}_{2}} w_{i}$, where $C>0$ depends only on $n, \sigma, A_{1}, A_2, \bar{r},$ and $\sup _{i}\|a_{i}\|_{L^{\infty}(B_{\bar{r}}(x_{i}))}$. Lemma \ref{lem:Harnack} follows after rescaling back to $u_{i}$.
\end{proof}

\begin{proposition}\label{pro:Blowup1}
Assume the assumptions in Lemma \ref{lem:Harnack}. Suppose that $\|a_{i}\|_{C^{2}(B_{3})} \leq A_{0} .$ Then for any $R_{i} \rightarrow \infty$ and $\varepsilon_{i} \rightarrow 0^{+}$, we have, after passing to a subsequence (still denoted as $\{u_{i}\}$, $\{x_{i}\}$, etc.), that
\begin{equation}\label{eq:B1-1}
\|m_{i}^{-1} u_{i}(m_{i}^{-(p_{i}-1) / 2 \sigma} \cdot+x_{i})-(1+\bar{c}|\cdot|^{2})^{(2 \sigma-n) / 2}\|_{C^{2}(B_{2 R_{i}}(0))} \leq \varepsilon_{i},
\end{equation}
\begin{equation}\label{eq:B1-2}
r_i:=R_{i} m_{i}^{-(p_{i}-1)/2 \sigma} \rightarrow 0 \quad \text { as } \, i \rightarrow \infty,
\end{equation}
where $m_{i}=u_{i}(x_{i})$ and $\bar{c}$ depends only on $n$ and $\sigma$.
\end{proposition}

\begin{proof}
Define
$$
\varphi_{i}(x):=m_{i}^{-1} u_{i}(m_{i}^{-(p_{i}-1) / 2 \sigma} x+x_{i})\quad \text { for } |x|<2m_{i}^{(p_{i}-1) / 2 \sigma},
$$
and
$$
\hat{h}_{i}(x):=\int_{B_{3} \backslash B_{2}}  \frac{a_i(y) u_i(y)+u_i(y)^{p_i}}{|x-y|^{n-2\sigma}} \,\ud y+h_{i}(x).
$$
It is easy to see that $\hat{h}_{i}$ also satisfies \eqref{eq:Harnackhi} for all $x \in B_{1}$. By the equation of $u_i$, we have
\begin{equation}\label{eq:B1-3}
\varphi_{i}(x)=\int_{B_{2m_{i}^{(p_{i}-1) / 2 \sigma}}} \frac{m_{i}^{1-p_{i}}a_{i}(m_{i}^{-(p_{i}-1) / 2 \sigma} {y}+x_{i}) \varphi_{i}({y})+\varphi_{i}({y})^{p_{i}}}{|x-y|^{n-2 \sigma}} \,\ud y+\tilde{h}_{i}(x),
\end{equation}
\begin{equation}\label{eq:B1-4}
\varphi_{i}(0)=1, \quad \nabla \varphi_{i}(0)=0, \quad 0<\varphi_{i}(x) \leq A_{1}|x|^{-2 \sigma /(p_{i}-1)},
\end{equation}
where $\tilde{h}_{i}(x):=m_{i}^{-1} \hat{h}_{i}(m_{i}^{-(p_{i}-1) / 2 \sigma} x+x_{i})$.

Since $\max _{\partial B_{1}} \hat{h}_{i} \leq \max _{\partial B_{1}} u_{i} \leq A_{3}$, by \eqref{eq:Harnackhi} for $\hat{h}_{i}$ we have
\begin{equation}\label{eq:tildehto0}
\tilde{h}_{i} \rightarrow 0 \quad \text { in } C_{l o c}^{2}(\mathbb{R}^{n}) \quad \text { as }\, i \rightarrow \infty.
\end{equation}
For any $R>0,$ we claim that
\begin{equation}\label{eq:B1-5}
\|\varphi_{i}\|_{C^{2, \alpha}(B_{R})} \leq C(R)
\end{equation}
for some $\alpha \in(0,1)$ and all sufficiently large $i$. Indeed, by \eqref{eq:B1-4} and Proposition \ref{pro:JLXSchauder}, it is sufficient to prove that $\varphi_{i} \leq C$ in $B_{1}$. If $\varphi_{i}(\bar{x}_{i})=\sup _{B_{1}} \varphi_{i} \rightarrow \infty,$ let
$$
\tilde{\varphi}_{i}(z):=\varphi_{i}(\bar{x}_{i})^{-1} \varphi_{i}(\varphi_{i}(\bar{x}_{i})^{-(p_{i}-1) / 2 \sigma} z+\bar{x}_{i}) \leq 1 \quad \text { for }\, |z| \leq \frac{1}{2} \varphi_{i}(\bar{x}_{i})^{(p_{i}-1) / 2 \sigma}.
$$
By \eqref{eq:B1-4},
$$
\tilde{\varphi}_{i}(z_{i})=\varphi_{i}(\bar{x}_{i})^{-1} \varphi_{i}(0) \rightarrow 0
$$
for $z_{i}=-\varphi_{i}(\bar{x}_{i})^{(p_{i}-1) / 2 \sigma} \bar{x}_{i} .$ Since $\varphi_{i}(\bar{x}_{i}) \leq A_{1}|\bar{x}_{i}|^{-2 \sigma /(p_{i}-1)},$ we have $|z_{i}| \leq A_{1}^{(p_{i}-1) /2 \sigma} .$ Hence, we can find $t>0$ independent of $i$ such that such that $z_{i} \in B_{t}$. Since $\tilde{\varphi}_{i}$ satisfies a similar equation to \eqref{eq:B1-3}, applying Proposition \ref{pro:JLXHarnack} to $\tilde{\varphi}_{i}$ in $B_{2 t}$, we conclude that
$$
1=\tilde{\varphi}_{i}(0) \leq C \tilde{\varphi}_{i}(z_{i}) \rightarrow 0,
$$
which is impossible. Thus \eqref{eq:B1-5} is valid.

It follows from \eqref{eq:B1-5} that, after passing to a subsequence if necessary,
\begin{equation}\label{eq:varitovar}
\varphi_{i} \rightarrow \varphi\quad \text{ in }\, C_{l o c}^{2}(\mathbb{R}^{n})
\end{equation}
for some $\varphi \in C^{2}(\mathbb{R}^{n}) .$ For any fixed $x$ and $R>2|x|$, by \eqref{eq:B1-4} we have
\begin{equation}\label{eq:O(R)}
\begin{aligned}
&\int_{B_{2 m_{i}^{(p_{i}-1) / 2\sigma}}\backslash B_R}\frac{m_{i}^{1-p_{i}}a_{i}(m_{i}^{-(p_{i}-1) / 2 \sigma} {y}+x_{i}) \varphi_{i}({y})+\varphi_{i}({y})^{p_{i}}}{|x-y|^{n-2 \sigma}} \,\ud y\\
\leq& C  \int_{B_{2 m_{i}^{(p_{i}-1) / 2\sigma}}\backslash B_R} |y|^{2\sigma-n}(m_{i}^{1-p_{i}}|y|^{-\frac{2\sigma}{p_i-1}}+|y|^{-\frac{2\sigma p_{i}}{p_{i}-1}}) \,\ud y\\
=&m_{i}^{1-p_{i}}O(R^{-\frac{n-6\sigma}{2}+O(\tau_{i})})+O(R^{-\frac{n-2\sigma}{2}+O(\tau_{i})}).
\end{aligned}
\end{equation}
Combining \eqref{eq:tildehto0}, \eqref{eq:varitovar} and \eqref{eq:O(R)} together, by \eqref{eq:B1-3} we have that for any fixed $R>0$ and $x \in B_{R / 2}$,
\begin{equation}
\varphi(x)=\int_{B_{R}} \frac{\varphi(y)^{\frac{n+2\sigma}{n-2\sigma}}}{|x-y|^{n-4}} \,\ud y+O(R^{-\frac{n-2\sigma}{2}}).
\end{equation}
Sending $R \rightarrow \infty$, it follows from Lebesgue's monotone convergence theorem that
$$
\varphi(x)=\int_{\mathbb{R}^{n}} \frac{\varphi(y)^{\frac{n+2\sigma}{n-2\sigma}}}{|x-y|^{n-4}} \,\ud y, \quad x \in \mathbb{R}^{n}.
$$
It follows the classification theorem in \cite{CLOClassification2006} or \cite{LRemark2004} that
$$
\varphi(x)=(1+\bar{c}|x|^{2})^{-\frac{n-2 \sigma}{2}}
$$
with $\bar{c}=(\frac{\pi^{n / 2} \Gamma(\sigma)}{\Gamma({n}/{2}+\sigma)})^{1 / \sigma},$ where we used that $\varphi(0)=1$ and $\nabla \varphi(0)=0$. Proposition \ref{pro:Blowup1} follows immediately.
\end{proof}

Note that since passing to subsequences does not affect our proofs, in the rest of the paper we will always choose $R_{i} \rightarrow \infty$ with $R_i^{\tau_i}=1+o(1)$ first, and then $\varepsilon_{i} \rightarrow 0^{+}$ as small as we wish (depending on $R_{i}$) and then choose our subsequence $\{u_{i}\}$ to work with.

Without loss of generality, we assume $\bar{r}=2$ to the end of the section.

\begin{proposition}\label{pro:Blowup2}
Under the hypotheses of Lemma \ref{lem:Harnack}, there exists some positive constant $C=C(n, \sigma, A_0, A_{1})$ such that,
$$
u_{i}(x) \geq C^{-1} m_{i}(1+\bar{c} m_{i}^{(p_{i}-1) / \sigma}|x-x_{i}|^{2})^{(2 \sigma-n) / 2}, \quad |x-x_{i}| \leq 1.
$$
In particular, for any $e \in \mathbb{R}^{n}$ with $|e|=1,$ we have
$$
u_{i}(x_{i}+e) \geq C^{-1} m_{i}^{-1+(n-2 \sigma)\tau_{i} / 2 \sigma },
$$
where $\tau_{i}=(n+2 \sigma) /(n-2 \sigma)-p_{i}$.
\end{proposition}

\begin{proof}
By change of variables and using Proposition \ref{pro:Blowup1}, some calculation leads to that, for $r_{i} \leq|x-x_{i}| \leq 1$,
\begin{equation}\label{eq:B2-1}
u_i(x) \geq \int_{|y-x_{i}| \leq r_{i}} \frac{u_{i}(y)^{p_{i}}}{|x-y|^{n-2 \sigma}} \,\ud y \geq \frac{1}{4} C m_{i} U(m_{i}^{(p_{i}-1) / 2 \sigma}(x-x_{i})),
\end{equation}
where $U(z)=(1+\bar{c}|z|^{2})^{(2 \sigma-n) / 2}$. The proposition follows immediately from the above and Proposition \ref{pro:Blowup1}.
\end{proof}

\begin{lemma}\label{lem:Blowup3}
Under the hypotheses of Lemma \ref{lem:Harnack}, and in addition that $x_{i} \rightarrow 0$ is also an isolated simple blow up point with the constant $\rho$, there exists $\delta_{i}>0$, $\delta_{i}=O(R_{i}^{-2 \sigma+o(1)}),$ such that
$$
u_{i}(x) \leq C u_{i}(x_{i})^{-\lambda_{i}}|x-x_{i}|^{2 \sigma-n+\delta_i} \quad \text { for all }\, r_{i} \leq|x-x_{i}| \leq 1,
$$
where $\lambda_{i}=(n-2 \sigma-\delta_i)(p_{i}-1) / 2 \sigma-1$ and $C>0$ depends only on $n, \sigma, A_{0}, A_{1}$ and $\rho$.
\end{lemma}

\begin{proof}
The proof will be divided into four steps.

{\bf Step 1.}\, From Proposition \ref{pro:Blowup1}, we have
\begin{equation}\label{eq:B3-1}
u_{i}(x) \leq C u_{i}(x_{i}) R_{i}^{2 \sigma-n} \quad \text { for all }\,|x-x_{i}|=r_{i}=R_{i} m_{i}^{-(p_{i}-1)/2 \sigma}.
\end{equation}
Let $\bar{u}_{i}(r)$ be defined as \eqref{22}. It follows from the assumption of isolated simple blow up points and Proposition \ref{pro:Blowup1} that
\begin{equation}\label{eq:B3-2}
r^{2 \sigma /(p_{i}-1)} \bar{u}_{i}(r)\, \text { is strictly decreasing for }\, r_{i}<r<\rho.
\end{equation}
By Lemma \ref{lem:Harnack}, \eqref{eq:B3-1} and \eqref{eq:B3-2}, we have, for all $r_{i}<|x-x_{i}|<\rho$,
\be\label{eq:sameargu}
\begin{aligned}
|x-x_{i}|^{2 \sigma /(p_{i}-1)} u_{i}(x) & \leq C|x-x_{i}|^{2 \sigma /(p_{i}-1)} \bar{u}_{i}(|x-x_{i}|) \\
& \leq C r_{i}^{2 \sigma /(p_{i}-1)} \bar{u}_{i}(r_{i})\\
&\leq C R_{i}^{(2 \sigma-n)/2},
\end{aligned}
\ee
where we used $R_{i}^{\tau_{i}}=1+o(1)$. Thus
\begin{equation}\label{eq:B3-3}
u_{i}(x)^{p_{i}-1} \leq C R_{i}^{-2 \sigma}|x-x_{i}|^{-2 \sigma}\quad \text { for all }\, r_{i} \leq|x-x_{i}|<\rho.
\end{equation}

{\bf Step 2.}\, Define
$$
\mathfrak{L}_i\phi(y):=\int_{B_3} \frac{[a_{i}(z)+u_{i}(z)^{p_{i}-1}] \phi(z)}{|y-z|^{n-2 \sigma}} \,\ud z.
$$
Thus
$$
u_{i}=\mathfrak{L}_{i} u_{i}+h_i.
$$
Note that for $2 \sigma<\mu<n$ and $0<|x|<2$,
\begin{equation}\label{eq:usefuline}
\begin{aligned}
\int_{B_3} \frac{1}{|x-y|^{n-2 \sigma}|y|^{\mu}} \,\ud y&\leq\int_{\mathbb{R}^n} \frac{1}{|x-y|^{n-2 \sigma}|y|^{\mu}} \,\ud y\\
&=|x|^{2 \sigma-n} \int_{\mathbb{R}^n} \frac{1}{||x|^{-1} x-|x|^{-1} y|^{n-2 \sigma}|y|^{\mu}} \,\ud y \\
&=|x|^{-\mu+2 \sigma} \int_{\mathbb{R}^{n}} \frac{1}{||x|^{-1} x-z|^{n-2 \sigma}|z|^{\mu}}\,\ud z \\
& \leq C\Big(\frac{1}{n-\mu}+\frac{1}{\mu-2 \sigma}+1\Big)|x|^{-\mu+2 \sigma},
\end{aligned}
\end{equation}
where we did the change of variables $y=|x| z .$ By \eqref{eq:B3-3}, we have
$$
\begin{aligned}
&\int_{r_{i}<|y-x_{i}|<\rho} \frac{[a_{i}(y)+u_{i}(y)^{p_{i}-1}]|y-x_{i}|^{-\delta_{i}}}{|x-y|^{n-2 \sigma}} \,\ud y\\
\leq&C\int_{\mathbb{R}^n} \frac{1}{|x-y|^{n-2 \sigma}|y-x_{i}|^{\delta_i}} \,\ud y+CR_{i}^{-2 \sigma}\int_{\mathbb{R}^n} \frac{1}{|x-y|^{n-2 \sigma}|y-x_{i}|^{2 \sigma+\delta_i}} \,\ud y\\
=&C\int_{\mathbb{R}^n} \frac{1}{|(x-x_i)-z|^{n-2 \sigma}|z|^{\delta_i}} \,\ud z+CR_{i}^{-2 \sigma}\int_{\mathbb{R}^n} \frac{1}{|(x-x_i)-z|^{n-2 \sigma}|z|^{2 \sigma+\delta_i}} \,\ud z\\
\leq&C|x-x_i|^{-\delta_i+2\sigma}+CR_{i}^{-2 \sigma}|x-x_i|^{-\delta_i}.
\end{aligned}
$$
Similarly,
$$
\int_{r_{i}<|y-x_{i}|<\rho} \frac{[a_{i}(y)+u_{i}(y)^{p_{i}-1}]|y-x_{i}|^{2 \sigma-n+\delta_{i}}}{|x-y|^{n-2 \sigma}} \,\ud y
\leq C|x-x_{i}|^{4 \sigma-n+\delta_{i}}+C R_{i}^{-2 \sigma}|x-x_{i}|^{2 \sigma-n+\delta_{i}},
$$
$$
\int_{r_{i}<|y-x_{i}|<\rho} \frac{[a_{i}(y)+u_{i}(y)^{p_{i}-1}]|y-x_{i}|^{4 \sigma-n}}{|x-y|^{n-2 \sigma}} \,\ud y
\leq C|x-x_i|^{6\sigma-n}+CR_{i}^{-2 \sigma}|x-x_i|^{4\sigma-n}.
$$
Choosing $0< \rho_1\leq \rho$ and $0<\delta_{i}=O(R_{i}^{-2 \sigma})$, such that for $r_i\leq |x-x_i|<\rho_1$,
\begin{equation}\label{eq:B3-4}
\int_{r_{i}<|y-x_{i}|<\rho_1} \frac{[a_{i}(y)+u_{i}(y)^{p_{i}-1}]|y-x_{i}|^{-\delta_{i}}}{|x-y|^{n-2 \sigma}} \,\ud y \leq \frac{1}{4}|x-x_{i}|^{-\delta_{i}},
\end{equation}
\begin{equation}\label{eq:B3-5}
\int_{r_{i}<|y-x_{i}|<\rho_1} \frac{[a_{i}(y)+u_{i}(y)^{p_{i}-1}]|y-x_{i}|^{2 \sigma-n+\delta_{i}}}{|x-y|^{n-2 \sigma}} \,\ud y \leq \frac{1}{4}|x-x_{i}|^{2 \sigma-n+\delta_{i}},
\end{equation}
and
\begin{equation}\label{eq:B3-6}
\int_{r_{i}<|y-x_{i}|<\rho_1} \frac{[a_{i}(y)+u_{i}(y)^{p_{i}-1}]|y-x_{i}|^{4 \sigma-n}}{|x-y|^{n-2 \sigma}} \,\ud y\leq \frac{1}{4}|x-x_{i}|^{4 \sigma-n}.
\end{equation}

Set
$$
M_{i}:=4 \cdot 2^{n-2 \sigma} \max \limits_{\partial B_{\rho_1}(x_{i})} u_{i}+2\max\limits_{\bar{B}_{\rho_1}(x_i)}h_i,
$$
$$
f_{i}(x):=M_{i} \rho_1^{\delta_{i}}|x-x_{i}|^{-\delta_{i}}+A m_{i}^{-\lambda_{i}}|x-x_{i}|^{2 \sigma-n+\delta_{i}}+Bm_i^{\frac{n-(n-2\sigma)p_i}{2\sigma}}|x-x_{i}|^{4\sigma-n},
$$
and
$$
\phi_{i}(x)=\left\{\begin{array}{ll}
f_{i}(x),\quad & r_{i}<|x-x_{i}|<\rho_1, \\
u_{i}(x),\quad & \text {otherwise},
\end{array}\right.
$$
where $A$ and $B$ will be chosen later.

By \eqref{eq:B3-4}, \eqref{eq:B3-5} and \eqref{eq:B3-6}, we have for $r_{i}<|x-x_{i}|<\rho_1$,
$$
\begin{aligned}
&\mathfrak{L}_{i} \phi_{i}(x)+h_i(x) \\
=&\int_{|y-x_{i}| \leq r_{i}}+\int_{r_{i}<|y-x_{i}|<\rho_1}+\int_{\rho_1 \leq|y-x_{i}|\leq 3} \frac{[a_{i}(y)+u_{i}(y)^{p_{i}-1}] \phi_{i}(y)}{|x-y|^{n-2 \sigma}}\,\ud y +h_i(x)\\
\leq& \int_{|y-x_{i}| \leq r_{i}} \frac{a_{i}(y)u_{i}(y)+u_{i}(y)^{p_{i}}}{|x-y|^{n-2 \sigma}} \,\ud y+\frac{f_{i}}{4}+\int_{\rho_1 \leq|y-x_{i}|\leq 3} \frac{a_{i}(y) u_{i}(y)+u_{i}(y)^{p_{i}}}{|x-y|^{n-2 \sigma}} \,\ud y+h_i(x).
\end{aligned}
$$

To estimate the first term, we use change of variables, Proposition \ref{pro:Blowup1} and the computations in \eqref{eq:B2-1} that,
$$
\begin{aligned}
&\int_{|y-x_{i}| \leq r_{i}} \frac{a_{i}(y) u_{i}(y)+u_{i}(y)^{p_{i}}}{|x-y|^{n-2 \sigma}} \,\ud y \\
\leq& Cm_i^{2-p_i}\int_{|z| \leq R_{i}} \frac{U(z)}{|m_{i}^{(p_{i}-1) / 2 \sigma}(x-x_{i})-z|^{n-2 \sigma}} \,\ud z+ C m_{i} \int_{|z| \leq R_{i}} \frac{U(z)^{p_{i}}}{|m_{i}^{(p_{i}-1) / 2 \sigma}(x-x_{i})-z|^{n-2 \sigma}} \,\ud z \\
\leq&Cm_i^{\frac{n-(n-2\sigma)p_i}{2\sigma}}|x-x_{i}|^{4\sigma-n}+ C m_{i} U(m_{i}^{(p_{i}-1) / 2 \sigma}(x-x_{i})).
\end{aligned}
$$
Since $|x-x_{i}|>r_{i},$ it follows that
$$
\begin{aligned}
m_{i} U(m_{i}^{(p_{i}-1) / 2 \sigma}(x-x_{i})) & \leq C m_{i}^{1-(p_{i}-1)(n-2 \sigma) / 2 \sigma}|x-x_{i}|^{2 \sigma-n} \\
& \leq C m_{i}^{-\lambda_{i}}|x-x_{i}|^{2 \sigma-n+\delta_{i}}.
\end{aligned}
$$

To estimate the third term, we let $\bar{x}=\rho_1 \frac{x-x_{i}}{|x-x_{i}|} \in \partial B_{\rho_1}(x_{i}),$ and then
\begin{equation}\label{eq:B3-7}
\begin{aligned}
\int_{\rho_1 \leq|y-x_{i}|\leq3} \frac{a_{i}(y) u_{i}(y)+u_{i}(y)^{p_{i}}}{|x-y|^{n-2 \sigma}}  \,\ud y&=\int_{\rho_1 \leq|y-x_{i}|\leq3} \frac{|\bar{x}-y|^{n-2 \sigma}}{|x-y|^{n-2 \sigma}} \frac{a_{i}(y)u_{i}(y)+u_{i}(y)^{p_{i}}}{|\bar{x}-y|^{n-2 \sigma}} \,\ud y \\
& \leq 2^{n-2 \sigma} \int_{\rho_1 \leq|y-x_{i}|\leq3} \frac{a_{i}(y)u_{i}(y)+u_{i}(y)^{p_{i}}}{|\bar{x}-y|^{n-2 \sigma}} \,\ud y \\
&\leq 2^{n-2 \sigma} u_{i}(\bar{x})\leq 2^{n-2 \sigma} \max _{\partial B_{\rho_1}(x_{i})} u_{i}\leq M_{i} / 4,
\end{aligned}
\end{equation}
where we have used \eqref{eq:intui=} and the positivity of $h_{i}$.

Therefore, we conclude that
\begin{equation}\label{eq:B3-8}
\mathfrak{L}_{i} \phi_{i}(x)+h_i(x) \leq \phi_{i}(x)\quad \text { for all }\, r_{i} \leq|x-x_{i}| \leq \rho_1,
\end{equation}
by choosing $A, B$ large.

{\bf Step 3.}\, In view of \eqref{eq:B3-1}, we may choose $A$ large such that $f_{i} \geq u_{i}$ on $\partial B_{r_{i}}(x_{i})$. By the choice of $M_{i}$, we know that $f_{i} \geq u_{i}$ on $\partial B_{\rho_1}(x_{i}) .$ We claim that
\begin{equation}\label{eq:uleqphi}
u_{i} \leq \phi_{i} \quad\text { for all }\, r_i\leq |x-x_i|\leq \rho_1.
\end{equation}
Indeed, if not, let
$$
1<t_{i}:=\inf \{t>1: t \phi_{i} \geq u_{i} \, \text { in }\, r_i\leq |x-x_i|\leq \rho_1\}<\infty.
$$
By the definition of $\phi_i$, we have $t_{i} \phi_{i} > u_{i}$ in $B_{r_{i}}(x_{i}) \cup B_{\rho_1}^{c}(x_{i})$. By the continuity there exists $y_{i} \in B_{\rho_1}(x_{i}) \backslash \overline{B}_{r_{i}}(x_{i})$ such that
$$
0=t_{i} \phi_{i}(y_{i})-u_{i}(y_{i}) \geq \mathfrak{L}_{i}(t_{i} \phi_{i}-u_{i})(y_{i})+(t_i-1)h_i(y_i)>0.
$$
This is a contradiction. Therefore, the claim is proved.

{\bf Step 4.}\, By \eqref{eq:Harnackhi}, we have $\max _{\bar{B}_{\rho_1}(x_i)} h_{i} \leq A_{2} \max _{\partial {B}_{\rho_1}(x_i)} h_{i} \leq A_{2} \max _{\partial {B}_{\rho_1}(x_i)} u_{i} .$ Hence, $M_{i} \leq$
$C \max _{\partial B_{\rho_1}(x_i)} u_{i} .$ For $r_{i}<\theta<\rho_1$, the same arguments as in \eqref{eq:sameargu} yield
$$
\begin{aligned}
\rho_1^{2 \sigma /(p_{i}-1)} M_{i} & \leq C \rho_1^{2 \sigma /(p_{i}-1)} \bar{u}_{i}(\rho_1) \\
& \leq C \theta^{2 \sigma /(p_{i}-1)} \bar{u}_{i}(\theta) \\
& \leq C \theta^{2 \sigma /(p_{i}-1)}\{M_{i} \rho_1^{\delta_{i}}\theta^{-\delta_{i}}+A m_{i}^{-\lambda_{i}}\theta^{2 \sigma-n+\delta_{i}}+Bm_i^{\frac{n-(n-2\sigma)p_i}{2\sigma}}\theta^{4\sigma-n}\}.
\end{aligned}
$$
Choose $\theta=\theta(n, \sigma, \rho, A_{0}, A_{1})$ sufficiently small so that
$$
C \theta^{2 \sigma /(p_{i}-1)} \rho_1^{\delta_{i}} \theta^{-\delta_{i}} \leq \frac{1}{2} \rho_1^{2 \sigma /(p_{i}-1)}.
$$
It follows that
$$
M_{i} \leq C m_{i}^{-\lambda_{i}}.
$$
Together with \eqref{eq:uleqphi}, Lemma \ref{lem:Blowup3} holds when $r_i\leq|x-x_{i}| \leq \rho_{1}$. By Lemma \ref{lem:Harnack} it also holds when $\rho_{1} \leq|x-x_{i}| \leq 1$. Therefore, we complete the proof.
\end{proof}

Below we are going to improve the estimate in Lemma \ref{lem:Blowup3}. First, we prove a Pohozaev type identity.

\begin{proposition}\label{pro:Pohozaev}
Let $0\leq u \in C(\overline{B}_{R})$ be a solution of
$$
u(x)=\int_{B_{R}} \frac{a(y) u(y)+u(y)^{p}}{|x-y|^{n-2 \sigma}} \, \ud y+h(x),
$$
where $1<p \leq \frac{n+2 \sigma}{n-2 \sigma},$ $a(x)\in C^1(B_R)$, and $h(x) \in C^{1}(B_{R}),$ $\nabla h \in L^{1}(B_{R}) .$ Then
\begin{equation}\label{eq:Pohozaev}
\mathcal{P}_{\sigma}(0, R, u, p)+\mathcal{Q}_{\sigma}(0,R, u, p)=0,
\end{equation}
where
$$
\mathcal{P}_{\sigma}(0, R, u, p):=-\frac{n-2 \sigma}{2} \int_{B_{R}} (a(x) u(x)+u(x)^{p}) h(x) \, \ud x-\int_{B_{R}} x \nabla h(x) (a(x) u(x)+u(x)^{p}) \, \ud x,
$$
$$
\begin{aligned}
\mathcal{Q}_{\sigma}(0, R, u, p):=&\Big(\frac{n-2 \sigma}{2}-\frac{n}{p+1}\Big) \int_{B_{R}}  u(x)^{p+1} \, \ud x-\sigma\int_{B_R} a(x)u(x)^2\,\ud x\\
&-\frac{1}{2} \int_{B_{R}} x \nabla a(x) u(x)^{2} \, \ud x+\frac{R}{2} \int_{\partial B_{R}} a(x) u(x)^{2} \, \ud s+\frac{R}{p+1} \int_{\partial B_{R}}  u(x)^{p+1} \, \ud s.
\end{aligned}
$$
\end{proposition}

\begin{proof}
Note that
$$
\begin{aligned}
&\frac{1}{2} \int_{B_{R}} x a(x) \nabla u(x)^{2} \, \ud x+\frac{1}{p+1} \int_{B_{R}} x  \nabla u(x)^{p+1} \, \ud x\\
=&\int_{B_{R}} x (a(x) u(x)+u(x)^{p}) \nabla u(x) \, \ud x\\
=&(2 \sigma-n) \int_{B_{R}} x (a(x) u(x)+u(x)^{p}) \int_{B_{R}} \frac{(x-y) (a(y) u(y)+u(y)^{p})}{|x-y|^{n+2-2 \sigma}} \, \ud y \, \ud x \\
&+\int_{B_{R}} x (a(x) u(x)+u(x)^{p}) \nabla h(x) \, \ud x.
\end{aligned}
$$
By the divergence theorem, we have
$$
\int_{B_{R}} x a(x) \nabla u(x)^{2} \, \ud x=-\int_{B_{R}}(n a(x)+x \nabla a(x)) u(x)^{2} \, \ud  x+R \int_{\partial B_{R}} a(x) u(x)^{2} \, \ud s
$$
and
$$
\int_{B_{R}} x \nabla u(x)^{p+1} \, \ud x=-\int_{B_{R}}n u(x)^{p+1} \, \ud  x+R \int_{\partial B_{R}}  u(x)^{p+1} \, \ud s.
$$
By direct computations,
$$
\begin{aligned}
&\int_{B_{R}} x (a(x) u(x)+u(x)^{p}) \int_{B_{R}} \frac{(x-y) (a(y) u(y)+u(y)^{p})}{|x-y|^{n+2-2 \sigma}} \, \ud y \, \ud x \\
=&\frac{1}{2} \int_{B_{R}}  (a(x) u(x)+u(x)^{p}) \int_{B_{R}} \frac{(|x-y|^{2}+(|x|^{2}-|y|^{2})) (a(y) u(y)+u(y)^{p})}{|x-y|^{n+2-2 \sigma}} \, \ud y \, \ud x\\
=&\frac{1}{2} \int_{B_{R}}  (a(x) u(x)+u(x)^{p}) \int_{B_{R}} \frac{a(y) u(y)+u(y)^{p}}{|x-y|^{n-2 \sigma}} \, \ud y \, \ud x\\
=&\frac{1}{2} \int_{B_{R}}  (a(x) u(x)+u(x)^{p})(u(x)-h(x)) \, \ud x.
\end{aligned}
$$
This completes the proof.
\end{proof}

\begin{lemma}\label{lem:Blowup4}
Assume as in Lemma \ref{lem:Blowup3}, we have
$$
\tau_{i}=O(u_{i}(x_{i})^{-\min \{\frac{4\sigma}{n-2 \sigma}, 1\}+o(1)}),\, \text{ and thus }\, m_i^{\tau_{i}}=1+o(1).
$$
\end{lemma}

\begin{proof}
For $x \in B_{1}(x_i)$, we write the Eq. \eqref{eq:intui=} of $u_{i}$ as
\begin{equation}\label{eq:B4-2}
u_{i}(x)=\int_{B_{1}(x_i)} \frac{a_{i}(y) u_{i}(y)+u_{i}(y)^{p_{i}}}{|x-y|^{n-2 \sigma}} \,\ud y+b_{i}(x),
\end{equation}
where
$$
b_{i}(x):=Q_i(x)+h_i(x)=\int_{B_{3} \backslash B_{1}(x_i)} \frac{a_{i}(y) u_{i}(y)+u_{i}(y)^{p_{i}}}{|x-y|^{n-2 \sigma}} \,\ud y+h_i(x).
$$
By Lemma \ref{lem:Blowup3}, we have $u_{i}(x) \leq C m_{i}^{-\lambda_{i}}$ for all $x \in B_{3 / 2}(x_i) \backslash B_{1 / 2}(x_i)$. Hence, $b_{i}(x)=Q_i(x)+h_i(x) \leq u_{i}(x) \leq C m_{i}^{-1+o(1)}$ for any $x \in \partial B_{1}(x_i)$. It follows from \eqref{eq:Harnackhi} that
\be\label{eq:hinablahi}
\max _{\overline{B}_{2}(x_i)} h_{i} \leq C \min _{\partial B_{1}(x_i)} h_{i} \leq C m_{i}^{-1+o(1)}, \quad \max _{B_{1}(x_i)}|\nabla h_{i}| \leq C \max _{\overline{B}_{2}(x_i)} h_{i} \leq C m_{i}^{-1+o(1)}.
\ee
Since $a_{i}$ and $u_{i}$ are nonnegative, by \eqref{eq:hinablahi} and the computation in \eqref{eq:B3-7}, we have for any $x \in B_{1}(x_{i})$,
\be\label{eq:binablabi}
|b_{i}(x)|+|\nabla b_{i}(x)| \leq C m_{i}^{-1+o(1)}.
\ee
Indeed, when $|x-x_{i}|<7 / 8$, it is easy to see that
$$
|\nabla Q_{i}(x)| \leq C \int_{|y-x_{i}| \geq 1, |y|\leq3} \frac{a_{i}(y) u_{i}(y)+u_{i}(y)^{p_{i}}}{|x-y|^{n-2 \sigma+1}}\,\ud y \leq C \max _{\partial B_{1}(x_{i})} u_{i} \leq C m_{i}^{-1+o(1)}.
$$
When $7 / 8 \leq|x-x_{i}| \leq 1,$ then
$$
\begin{aligned}
|\nabla Q_{i}(x)| & \leq C \int_{|y-x_{i}| \geq 1, |y|\leq3, |y-x| \geq 1 / 8}+\int_{|y-x_{i}| \geq 1, |y|\leq3, |y-x|<1 / 8} \frac{a_{i}(y) u_{i}(y)+u_{i}(y)^{p_{i}}}{|x-y|^{n-2 \sigma+1}}\,\ud y \\
& \leq C \max _{\partial B_{1}(x_{i})} u_{i}+C (m_{i}^{-1+o(1)}+m_{i}^{-p_{i}+o(1)}) \int_{1-|x-x_i|<|y-x|<1 / 8} \frac{\,\ud y}{|x-y|^{n-2 \sigma+1}} \\
& \leq C m_{i}^{-1+o(1)}.
\end{aligned}
$$
Applying Proposition \ref{pro:Pohozaev} to \eqref{eq:B4-2} yields
\begin{equation}\label{eq:B4-5}
\begin{aligned}
&\tau_{i} \int_{B_{1}(x_{i})} u_{i}(x)^{p_{i}+1} \, \ud x \\
\leq & C\Big(\int_{B_{1}(x_{i})} u_{i}(x)^{2}+|x-x_{i}| u_{i}(x)^{2}+(b_{i}(x)+|\nabla b_{i}(x)|) (u_i(x)+u_{i}(x)^{p_{i}})\,\ud x\\
&+\int_{\partial B_{1}(x_{i})} u_i(x)^2+u_{i}(x)^{p_{i}+1} \, \ud s\Big).
\end{aligned}
\end{equation}
By Proposition \ref{pro:Blowup1} and \eqref{eq:binablabi}, we have
\begin{equation}\label{eq:B4-6}
\begin{aligned}
\int_{B_{1}(x_{i})} u_{i}(x)^{p_{i}+1}  \,\ud  x & \geq C \int_{B_{r_{i}}(x_{i})} \frac{m_{i}^{p_{i}+1}}{(1+\bar{c}|m_{i}^{(p_{i}-1) / 2 \sigma}(x-x_{i})|^{2})^{(n-2 \sigma)(p_{i}+1) / 2}} \,\ud x \\
& \geq C m_{i}^{\tau_{i}(n / 2 \sigma-1)} \int_{B_{R_{i}}} \frac{1}{(1+\bar{c}|z|^{2})^{(n-2 \sigma)(p_{i}+1) / 2}} \,\ud  z \\
& \geq C m_{i}^{\tau_{i}(n / 2 \sigma-1)},
\end{aligned}
\end{equation}
$$
\int_{B_{r_{i}}(x_{i})} u_{i}(x)^{2} \,\ud  x \leq C m_{i}^{-4\sigma/(n-2\sigma)+o(1)},
$$
$$
\int_{B_{r_{i}}(x_{i})}|x-x_{i}| u_{i}(x)^{2} \,\ud  x \leq \left\{\begin{array}{ll}
O(m_{i}^{-(4\sigma+2)/(n-2 \sigma)+o(1)}), \quad& n>4 \sigma+1, \\
O(m_{i}^{-(4\sigma+2)/(n-2 \sigma)+o(1)}) \ln m_{i}, \quad& n=4 \sigma+1, \\
o(m_{i}^{-2+o(1)}), \quad& n<4 \sigma+1,
\end{array}\right.
$$
$$
\int_{B_{r_{i}}(x_{i})} (b_{i}(x)+|\nabla b_{i}(x)|)u_{i}(x) \,\ud  x \leq C m_{i}^{-2n/(n-2\sigma)+o(1)},
$$
and
$$
\int_{B_{r_{i}}(x_{i})} (b_{i}(x)+|\nabla b_{i}(x)|)u_{i}(x)^{p_i} \,\ud  x \leq C m_{i}^{-2+o(1)}.
$$
By Lemma \ref{lem:Blowup3} and \eqref{eq:binablabi}, we have
$$
\int_{r_{i}<|x-x_{i}|<1}(b_{i}(x)+|\nabla b_{i}(x)|) u_{i}(x) \,\ud  x \leq C m_{i}^{-2+o(1)},
$$
$$
\int_{r_{i}<|x-x_{i}|<1}(b_{i}(x)+|\nabla b_{i}(x)|) u_{i}(x)^{p_{i}} \,\ud  x \leq C m_{i}^{-2+o(1)},
$$
$$
\int_{r_{i}<|x-x_{i}|<1} u_{i}(x)^{2} \,\ud  x \leq \left\{\begin{array}{ll}
O(m_{i}^{-2 \lambda_{i}}), \quad& n<2(2 \sigma+\delta_i), \\
O(m_{i}^{-2 \lambda_{i}})\ln m_{i}, \quad& n=2(2 \sigma+\delta_i), \\
O(m_{i}^{-4\sigma/(n-2 \sigma)+o(1)}), \quad& n>2(2 \sigma+\delta_i),
\end{array}\right.
$$
$$
\int_{r_{i}<|x-x_{i}|<1}|x-x_{i}| u_{i}(x)^{2} \,\ud  x \leq \left\{\begin{array}{ll}
O(m_{i}^{-2 \lambda_{i}}), \quad& n<2(2 \sigma+\delta_i)+1, \\
O(m_{i}^{-2 \lambda_{i}})\ln m_{i}, \quad& n=2(2 \sigma+\delta_i)+1, \\
O(m_{i}^{-(4\sigma+2)/(n-2 \sigma)+o(1)}), \quad& n>2(2 \sigma+\delta_i)+1,
\end{array}\right.
$$
$$
\int_{\partial B_{1}(x_{i})} u_{i}(x)^{2} \,\ud  s \leq C m_{i}^{-2+o(1)},
$$
and
$$
\int_{\partial B_{1}(x_{i})} u_{i}(x)^{p_{i}+1} \,\ud  s \leq C m_{i}^{-2 n /(n-2 \sigma)+o(1)}.
$$
Combining the above estimates and $\tau_{i}=o(1)$, we complete the proof.
\end{proof}

\begin{proposition}\label{pro:Blowup6}
Under the assumptions of Lemma \ref{lem:Blowup3}, we have
$$
u_{i}(x) \leq C u_{i}(x_{i})^{-1}|x-x_{i}|^{2 \sigma-n}\quad \text { for all }\, |x-x_{i}| \leq 1.
$$
\end{proposition}

\begin{proof}
For $|x-x_{i}| \leq r_{i},$ it follows from Proposition \ref{pro:Blowup1} that
$$
\begin{aligned}
u_{i}(x) \leq& C m_{i}\Big(\frac{1}{1+|m_{i}^{(p_{i}-1) / 2 \sigma}(x-x_{i})|^{2}}\Big)^{\frac{n-2\sigma}{2}} \\
\leq& C m_{i}^{-1+\frac{n-2 \sigma}{2 \sigma} \tau_{i}}|x-x_{i}|^{2 \sigma-n} \\
\leq& C m_{i}^{-1}|x-x_{i}|^{2 \sigma-n},
\end{aligned}
$$
where Lemma \ref{lem:Blowup4} is used in the last inequality.

We shall show first that
\begin{equation}\label{eq:B6-1}
u_{i}(\rho e+x_{i}) u_{i}(x_{i}) \leq C
\end{equation}
for any vector $|e|=1$. Since $u_{i}(x) \leq A_{1}|x-x_{i}|^{-2 \sigma /(p_{i}-1)}$ in $B_{2}(x_{i})$, it follows from Lemma \ref{lem:Harnack} that for any $0<\varepsilon<1$ there exists a positive constant $C(\varepsilon),$ depending on $n, \sigma, A_{0}, A_{1}$, and $\varepsilon,$ such that
\begin{equation}\label{eq:B6-2}
\sup _{B_{1}(x_{i}) \backslash B_{\varepsilon}(x_{i})} u_{i} \leq C(\varepsilon)  \inf _{B_{1}(x_{i}) \backslash B_{\varepsilon}(x_{i})} u_{i}.
\end{equation}
Define $\varphi_{i}(x):=u_{i}(\rho e+x_{i})^{-1} u_{i}(x) .$ Then for $|x-x_{i}| \leq 1$,
$$
\varphi_{i}(x)=\int_{B_{1}(x_i)} \frac{a_{i}(y) \varphi_{i}(y)+u_{i}(\rho e+x_{i})^{p_i-1}\varphi_{i}(y)^{p_{i}}}{|x-y|^{n-2 \sigma}} \, \ud y+\tilde{h}_{i}(x),
$$
where
$$
\tilde{h}_{i}(x)=\int_{B_3\backslash B_{1}(x_i)} \frac{a_{i}(y) \varphi_{i}(y)+u_{i}(\rho e+x_{i})^{p_i-1}\varphi_{i}(y)^{p_{i}}}{|x-y|^{n-2 \sigma}} \, \ud y+u_{i}(\rho e+x_{i})^{-1} h_{i}(x).
$$
Since $\varphi_{i}(\rho e+x_{i})=1,$ by \eqref{eq:B6-2},
\begin{equation}\label{eq:B6-3}
\|\varphi_{i}\|_{L^{\infty}(B_{1}(x_{i}) \backslash B_{\varepsilon}(x_{i}))} \leq C(\varepsilon) \quad\text { for }\, 0<\varepsilon<1.
\end{equation}
By \eqref{eq:Harnackhi},
\be\label{eq:B6-3.5}
\sum_{k=1}^{3}\|\nabla^{k} \tilde{h}_{i}\|_{L^{\infty}(B_{1}(x_i))} \leq C\|\tilde{h}_{i}\|_{L^{\infty}(B_{5 / 4}(x_i))} \leq C \tilde{h}_{i}(\rho e+x_i) \leq C.
\ee
Besides, by Lemma \ref{lem:Blowup3},
\begin{equation}\label{eq:B6-4}
u_{i}(\rho e+x_{i})^{p_{i}-1} \rightarrow 0 \quad\text { as }\, i \rightarrow \infty.
\end{equation}
Applying Proposition \ref{pro:JLXHarnack} and Proposition \ref{pro:JLXSchauder} to $\varphi_{i}$ and making use
of \eqref{eq:B6-3} and \eqref{eq:B6-3.5}, we have, after passing to a subsequence, that
$$
\tilde{h}_{i} \rightarrow h \quad \text { in }\, C^{2}(B_{1}), \quad \varphi_{i} \rightarrow \varphi \quad \text { in }\, C_{l o c}^{2}(B_{1} \backslash\{0\})
$$
for some $h \in C^{2}(B_{1})$ and $\varphi \in C_{loc}^{2}(B_{1} \backslash\{0\})$.

Therefore,
\be\label{eq:var-h}
\begin{aligned}
&\int_{B_{1}(x_i)} \frac{a_{i}(y) \varphi_{i}(y)+u_{i}(\rho e+x_{i})^{p_i-1}\varphi_{i}(y)^{p_{i}}}{|x-y|^{n-2 \sigma}} \, \ud y\\
=&\varphi_{i}(x)-\tilde{h}_{i}(x) \rightarrow \varphi(x)-h(x) \quad \text { in }\,  C_{l o c}^{2}(B_{1} \backslash\{0\}).
\end{aligned}
\ee
We shall evaluate what $G(x):=\varphi(x)-h(x)$ is. For any $|x|>0$ and $2|x_{i}|<r_i \leq \varepsilon<\frac{1}{2}|x-x_{i}|,$ in view of \eqref{eq:B6-3} and \eqref{eq:B6-4} we have
$$
\begin{aligned}
&\int_{B_{1}(x_{i})} \frac{a_{i}(y) \varphi_{i}(y)+u_{i}(\rho e+x_{i})^{p_i-1}\varphi_{i}(y)^{p_{i}}}{|x-y|^{n-2 \sigma}} \,\ud y \\
=& \int_{B_{\varepsilon}(x_{i})} \frac{u_{i}(\rho e+x_{i})^{p_i-1}\varphi_{i}(y)^{p_{i}}}{|x-y|^{n-2 \sigma}} \,\ud y + \int_{B_{1}(x_{i})} \frac{a_{i}(y) \varphi_{i}(y)}{|x-y|^{n-2 \sigma}} \,\ud y+o(1)\\
=&\int_{B_{\varepsilon}(x_{i})} \frac{u_{i}(\rho e+x_{i})^{-1}u_{i}(y)^{p_{i}}}{|x-y|^{n-2 \sigma}} \,\ud y + \int_{B_{1}(x_{i})} \frac{a_{i}(y)u_{i}(\rho e+x_{i})^{-1} u_{i}(y)}{|x-y|^{n-2 \sigma}} \,\ud y+o(1)\\
=&u_{i}(\rho e+x_{i})^{-1}\Big(|x|^{2\sigma-n}\int_{B_{\varepsilon}(x_{i})} {u_{i}(y)^{p_{i}}} \,\ud y+\int_{B_{\varepsilon}(x_{i})} O(|y|){u_{i}(y)^{p_{i}}} \,\ud y\\
&+ |x|^{2\sigma-n}\int_{B_{1}(x_{i})} {a_{i}(y) u_{i}(y)} \,\ud y+\int_{B_{1}(x_{i})}O(|y|) {a_{i}(y) u_{i}(y)} \,\ud y\Big)+o(1).
\end{aligned}
$$
By Proposition \ref{pro:Blowup1}, Lemma \ref{lem:Blowup4}, and note that $|y|\leq|y-x_i|+|x_i|$, we have
$$
m_i\int_{B_{r_i}(x_{i})} u_{i}(y)^{p_{i}} \,\ud y\to \int_{\mathbb{R}^n} (1+\bar{c}|z|^2)^{-(n+2\sigma)/2}\,\ud z,
$$
$$
m_i\int_{B_{r_i}(x_{i})} O(|y|){u_{i}(y)^{p_{i}}} \,\ud y\to 0,
$$
$$
m_i\int_{B_{r_i}(x_{i})} {a_{i}(y) u_{i}(y)} \,\ud y\to 0,
$$
and
$$
m_i\int_{B_{r_i}(x_{i})}O(|y|) {a_{i}(y) u_{i}(y)} \,\ud y\to 0
$$
as $i\to\infty$. By Lemma \ref{lem:Blowup3} and Lemma \ref{lem:Blowup4}, we have
$$
m_i\int_{B_{\varepsilon}(x_i)\backslash B_{r_i}(x_{i})} u_{i}(y)^{p_{i}} \,\ud y\to 0,
$$
and
$$
m_i\int_{B_{\varepsilon}(x_i)\backslash B_{r_i}(x_{i})} O(|y|){u_{i}(y)^{p_{i}}} \,\ud y\to 0
$$
as $i\to\infty$. Moreover, choosing $R_i$ such that $m_i^{\delta_i}\leq C$, we have, after passing to a subsequence,
$$
m_i\int_{B_{1}(x_i)\backslash B_{r_i}(x_{i})} {a_{i}(y) u_{i}(y)} \,\ud y\to C_1,
$$
and
$$
m_i\int_{B_{1}(x_i)\backslash B_{r_i}(x_{i})}O(|y|) {a_{i}(y) u_{i}(y)} \,\ud y\to C_2
$$
as $i\to\infty$, where $C_1$, $C_2$ are some nonegative constants.

Therefore, we have
$$
\int_{B_{1}(x_{i})} \frac{a_{i}(y) \varphi_{i}(y)+u_{i}(\rho e+x_{i})^{p_i-1}\varphi_{i}(y)^{p_{i}}}{|x-y|^{n-2 \sigma}} \,\ud y
=\frac{\bar{C}|x|^{2 \sigma-n}+C_{2}}{u_{i}(\rho e+x_{i}) u_{i}(x_{i})}+o(1).
$$
Since $x_{i} \rightarrow 0$ is an isolated simple blow point, we have $r^{\frac{n-2 \sigma}{2}} \bar{\varphi}(r) \geq \rho^{\frac{n-2 \sigma}{2}} \bar{\varphi}(\rho)$ for $0<r<\rho$. Consequently, $\varphi$ is singular at 0. By \eqref{eq:B6-3.5}, $\varphi(\bar{x})-h(\bar{x})\geq C>0$ at some point $\bar{x}\in B_1$. It follows from \eqref{eq:var-h} that
$$
u_{i}(\rho e+x_{i}) u_{i}(x_{i})\to \frac{1}{\varphi(\bar{x})-h(\bar{x})}(\bar{C}|\bar{x}|^{2 \sigma-n}+C_{2})\quad \text{ as }\, i\to\infty.
$$
Therefore, we have
\be\label{eq:uiuileqC}
u_{i}(\rho e+x_{i}) u_{i}(x_{i})\leq C, \quad \forall\, |e|=1.
\ee

Without loss of generality, we may assume that $\rho \leq 1 / 2$. It follows from Lemma \ref{lem:Harnack} and \eqref{eq:uiuileqC} that Proposition \ref{pro:Blowup6} holds for $\rho \leq|x-x_{i}| \leq 1$. By a standard scaling argument, we can reduce the case of $r_{i} \leq|x-x_{i}| \leq \rho$ to $|x-x_{i}|=1$. We refer to \cite[p. 132]{JLXThe2017} for details.
\end{proof}

\begin{corollary}\label{cor:Blowup7}
Under the hypotheses of Lemma \ref{lem:Blowup3}, we have
$$
\int_{|x-x_{i}| \leq 1}|x-x_{i}|^{s} u_{i}(x)^{2} \,\ud x=\left\{\begin{array}{ll}
O(m_{i}^{-2}), \quad& s+4 \sigma>n, \\
O(m_{i}^{-2}) \ln m_{i}, \quad& s+4 \sigma=n, \\
O(m_{i}^{-(4 \sigma+2 s)/(n-2 \sigma)}), \quad& s+4 \sigma<n.
\end{array}\right.
$$
\end{corollary}

\begin{proof}
It follows from Proposition \ref{pro:Blowup1}, Lemma \ref{lem:Blowup4} and Proposition \ref{pro:Blowup6}.
\end{proof}

\begin{lemma}\label{lem:Blowup9}
Under the hypotheses of Lemma \ref{lem:Blowup3}, we have
$$
a_{i}(x_{i}) \leq C\left\{\begin{array}{ll}
(\ln m_{i})^{-1}(1+\|\nabla^{2} a_{i}\|_{L^{\infty}(B_{1})}) \quad& \text { if }\, n=4 \sigma, \\
m_{i}^{-2+\frac{4 \sigma}{n-2 \sigma}}(1+\|\nabla^{2} a_{i}\|_{L^{\infty}(B_{1})}) \quad& \text { if }\, 4 \sigma<n<4 \sigma+2, \\
m_{i}^{-2+\frac{4 \sigma}{n-2 \sigma}}(1+\|\nabla^{2} a_{i}\|_{L^{\infty}(B_{1})} \ln m_{i}) \quad& \text { if }\, n=4 \sigma+2, \\
m_{i}^{-2+\frac{4 \sigma}{n-2 \sigma}}(1+\|\nabla^{2} a_{i}\|_{L^{\infty}(B_{1})} m_{i}^{\frac{2n-8\sigma-4}{n-2 \sigma}}) \quad& \text { if }\, n>4 \sigma+2,
\end{array}\right.
$$
and
$$
|\nabla a_{i}(x_{i})| \leq C\left\{\begin{array}{ll}
(\ln m_{i})^{-1}(1+\|\nabla^{2} a_{i}\|_{L^{\infty}(B_{1})}) \quad& \text { if }\, n=4 \sigma, \\
m_{i}^{-2+\frac{4 \sigma}{n-2 \sigma}}(1+\|\nabla^{2} a_{i}\|_{L^{\infty}(B_{1})}) \quad& \text { if }\, 4 \sigma<n<4 \sigma+1, \\
m_{i}^{-2+\frac{4 \sigma}{n-2 \sigma}}(1+\|\nabla^{2} a_{i}\|_{L^{\infty}(B_{1})} \ln m_{i}) \quad& \text { if }\, n=4 \sigma+1, \\
m_{i}^{-2+\frac{4 \sigma}{n-2 \sigma}}(1+\|\nabla^{2} a_{i}\|_{L^{\infty}(B_{1})} m_{i}^{\frac{2n-8\sigma-2}{n-2 \sigma}}) \quad& \text { if }\, n>4 \sigma+1,
\end{array}\right.
$$
where $C>0$ depends only on $n, \sigma, A_{0}, A_{1}$ and $\rho .$
\end{lemma}

\begin{proof}
It follows from Proposition \ref{pro:Blowup6} and the proof of \eqref{eq:binablabi} that for $y \in B_{1}(x_{i})$,
\begin{equation}\label{eq:B9-1}
|\nabla b_{i}(y)| \leq C m_{i}^{-1}.
\end{equation}
By \eqref{eq:B4-2} and integrating by parts, we have
$$
\begin{aligned}
&\frac{1}{2} \int_{\partial B_{1}(x_i)} y_{j} a_{i}(y)  u_{i}(y)^{2} \,\ud s-\frac{1}{2} \int_{B_{1}(x_i)} \partial_{j}a_{i}(y) u_{i}(y)^{2} \,\ud y+\frac{1}{p_{i}+1} \int_{\partial B_{1}(x_i)} y_{j}  u_{i}(y)^{p_{i}+1} \,\ud  s \\
=&\frac{1}{2}\int_{B_{1}(x_i)} a_{i}(y) \partial_{j} u_{i}(y)^{2} \,\ud  y+\frac{1}{p_{i}+1} \int_{B_{1}(x_i)}  \partial_{j} u_{i}(y)^{p_{i}+1} \,\ud  y \\
=&(2 \sigma-n) \int_{B_{1}(x_i)} a_{i}(y)  u_{i}(y) \int_{B_{1}(x_i)} \frac{(y_{j}-z_{j})}{|y-z|^{n-2 \sigma+2}} (a_{i}(z) u_{i}(z)+u_{i}(z)^{p_{i}}) \,\ud z  \,\ud y \\
&+\int_{B_{1}(x_i)} a_{i}(y)  u_{i}(y) \partial_{j} b_{i}(y) \,\ud  y \\
&+(2 \sigma-n) \int_{B_{1}(x_i)}  u_{i}(y)^{p_{i}} \int_{B_{1}(x_i)} \frac{(y_{j}-z_{j})}{|y-z|^{n-2 \sigma+2}} (a_{i}(z) u_{i}(z)+u_{i}(z)^{p_{i}})\,\ud z \,\ud  y \\
&+\int_{B_{1}(x_i)} u_{i}(y)^{p_{i}} \partial_{j} b_{i}(y) \,\ud  y \\
=&\int_{B_{1}(x_i)} (a_{i}(y) u_{i}(y)+u_{i}(y)^{p_{i}}) \partial_{j} b_{i}(y) \,\ud y.
\end{aligned}
$$
Hence,
$$
\begin{aligned}
&\Big| \int_{B_{1}(x_i)} \nabla a_{i}(y) u_{i}(y)^{2} \,\ud y\Big| \\
\leq& C\Big(\int_{B_{1}(x_i)}  (a_{i}(y) u_{i}(y)+u_{i}(y)^{p_{i}})|\nabla b_{i}(y)|\,\ud y+\int_{\partial B_{1}(x_i)} u_{i}(y)^{2}+u_{i}(y)^{p_{i}+1}\,\ud s\Big).
\end{aligned}
$$
Using Proposition \ref{pro:Blowup1}, Proposition \ref{pro:Blowup6} and \eqref{eq:B9-1}, we have
$$
\begin{aligned}
&\Big| \int_{B_{1}(x_i)} \nabla a_{i}(y) u_{i}(y)^{2} \,\ud y\Big|\\
\leq& Cm_i^{-2}+Cm_i^{-1}\int_{B_{1}(x_i)} a_{i}(y) u_{i}(y)+u_{i}(y)^{p_{i}}\,\ud y \\
\leq& Cm_i^{-2}.
\end{aligned}
$$
We write
$$
\begin{aligned}
&\int_{B_{1}(x_i)} \nabla a_{i}(y) u_{i}(y)^{2} \,\ud y \\
=&\nabla a_{i}(x_{i}) \int_{B_{1}(x_i)}  u_{i}(y)^{2} \,\ud y+\int_{B_{1}(x_i)}(\nabla a_{i}(y)-\nabla a_{i}(x_{i})) u_{i}(y)^{2} \,\ud y,
\end{aligned}
$$
by the triangle inequality and Corollary \ref{cor:Blowup7} we have
$$
\begin{aligned}
\Big|\nabla a_{i}(x_{i}) \int_{B_{1}(x_i)}  u_{i}(y)^{2} \,\ud y\Big| \leq & C m_i^{-2}+C \int_{B_{1}(x_i)}|\nabla a_{i}(y)-\nabla a_{i}(x_{i})| u_{i}(y)^{2} \,\ud y \\
\leq& C\left\{\begin{array}{ll}
m_{i}^{-2}(1+\|\nabla^{2} a_{i}\|_{L^{\infty}(B_{1})}) \quad& \text { if }\, n<4 \sigma+1, \\
m_{i}^{-2}(1+\|\nabla^{2} a_{i}\|_{L^{\infty}(B_{1})} \ln m_{i})\quad & \text { if }\, n=4 \sigma+1,\\
m_{i}^{-2}(1+\|\nabla^{2} a_{i}\|_{L^{\infty}(B_{1})} m_{i}^{(2n-8\sigma-2)/(n-2 \sigma)}) \quad& \text { if }\, n>4 \sigma+1.
\end{array}\right.
\end{aligned}
$$
By Proposition \ref{pro:Blowup2} and Lemma \ref{lem:Blowup4},
\begin{equation}\label{eq:B9-2}
\int_{B_1(x_i)} u_{i}(y)^{2}\,\ud y \geq C m_{i}^{-{4 \sigma}/(n-2 \sigma)} \int_{B_{m_i^{(p_i-1)/2\sigma}}}  \frac{1}{(1+\bar{c}|y|^{2})^{n-2 \sigma}} \,\ud y.
\end{equation}
Therefore, desired estimates of $|\nabla a_{i}(x_{i})|$ follows.

Using Proposition \ref{pro:Pohozaev} and Lemma \ref{lem:Blowup4}, by the same proof of Lemma \ref{lem:Blowup4} we have
\begin{equation}\label{eq:tauiestimate}
\begin{aligned}
\tau_i\leq&Cu_{i}(x_{i})^{-2}+  C\int_{B_{1}(x_{i})} |y-x_{i}| |\nabla a_{i}(y)|u_i(y)^2\,\ud y\\
\leq&Cu_{i}(x_{i})^{-2}+C|\nabla a_{i}(x_{i})| \int_{B_{1}(x_{i})}|y-x_{i}| u_{i}(y)^{2} \,\ud y \\
&+C\|\nabla^{2} a_{i}\|_{L^{\infty}(B_{1})} \int_{B_{1}(x_{i})}|y-x_{i}|^{2} u_{i}(y)^{2} \,\ud y\\
\leq&C|\nabla a_{i}(x_{i})| \int_{B_{1}(x_{i})}|y-x_{i}| u_{i}(y)^{2} \,\ud y \\
&+ C\left\{\begin{array}{ll}
m_{i}^{-2}(1+\|\nabla^{2} a_{i}\|_{L^{\infty}(B_{1})}) \quad& \text { if }\, n<4 \sigma+2, \\
m_{i}^{-2}(1+\|\nabla^{2} a_{i}\|_{L^{\infty}(B_{1})} \ln m_{i}) \quad& \text { if }\, n=4 \sigma+2, \\
m_{i}^{-2}(1+\|\nabla^{2} a_{i}\|_{L^{\infty}(B_{1})} m_{i}^{(2n-8 \sigma-4)/(n-2 \sigma)}) \quad& \text { if }\, n>4 \sigma+2,
\end{array}\right.\\
\leq&C\left\{\begin{array}{ll}
m_{i}^{-2}(1+\|\nabla^{2} a_{i}\|_{L^{\infty}(B_{1})}) \quad& \text { if }\, n<4 \sigma+2, \\
m_{i}^{-2}(1+\|\nabla^{2} a_{i}\|_{L^{\infty}(B_{1})} \ln m_{i}) \quad& \text { if }\, n=4 \sigma+2, \\
m_{i}^{-2}(1+\|\nabla^{2} a_{i}\|_{L^{\infty}(B_{1})} m_{i}^{(2n-8 \sigma-4)/(n-2 \sigma)}) \quad& \text { if }\, n>4 \sigma+2.
\end{array}\right.
\end{aligned}
\end{equation}
Using Proposition \ref{pro:Pohozaev} again, by the estimates for $|\nabla a_{i}(x_{i})|$, \eqref{eq:B9-2}, and \eqref{eq:tauiestimate}, the estimates of $a_{i}(x_{i})$ follows immediately.
\end{proof}

\section{Expansions of blow up solutions}
\begin{lemma}\label{lem:Expansions1}
For $\ell>100$, $0<\alpha<n$ and $\alpha \leq \mu$, we have
$$
\int_{|y|\leq \ell} \frac{1}{|x-y|^{n-\alpha}} \frac{1}{(1+|y|)^{\mu}} \,\ud y \leq C\left\{\begin{array}{ll}
\ln (\frac{\ell}{r}+1) & \text { if }\, \mu=\alpha, \\
(1+r)^{\alpha-\mu} & \text { if }\, \alpha<\mu<n, \\
(1+r)^{\alpha-n} \ln (2+r) & \text { if }\, \mu=n, \\
(1+r)^{\alpha-n} & \text { if }\, \mu>n,
\end{array}\right.
$$
for all $r=|x|<\ell$, where $C>0$ is independent of $\ell$.
\end{lemma}

\begin{proof}
By change of variables $y=r z$, we have
$$
\begin{aligned}
&\int_{|y|\leq \ell}  \frac{1}{|x-y|^{n-\alpha}} \frac{1}{(1+|y|)^{\mu}} \,\ud y \\
=&r^{\alpha} \int_{|z|\leq \ell/r} \frac{1}{|x / r-z|^{n-\alpha}} \frac{1}{(1+r|z|)^{\mu}} \,\ud  z \\
=&r^{\alpha} \int_{|z| \leq 1 / 10}+\int_{1/10\leq|z| \leq \ell/r} \frac{1}{|x / r-z|^{n-\alpha}} \frac{1}{(1+r|z|)^{\mu}} \,\ud  z \\
\leq& C r^{\alpha} \int_{|z| \leq 1 / 10} \frac{1}{(1+r|z|)^{\mu}} \,\ud  z+C r^{\alpha-\mu} \int_{1/10\leq|z| \leq \ell/r} \frac{1}{|x / r-z|^{n-\alpha}} \frac{1}{|z|^{\mu}} \,\ud  z.
\end{aligned}
$$
The lemma follows immediately.
\end{proof}

Let
$$
\Theta_{\lambda}(z)=\Big(\frac{\lambda}{1+\bar{c}\lambda^{2} |z|^{2}}\Big)^{\frac{n-2 \sigma}{2}},
$$
where $\bar{c}$ is chosen as in Proposition \ref{pro:Blowup1}. Notice that $\Theta_{\lambda}(z)$ satisfies
\begin{equation}\label{eq:Thetalam}
\Theta_{\lambda}(z)=\int_{\mathbb{R}^{n}} \frac{\Theta_{\lambda}(y)^{\frac{n+2 \sigma}{n-2 \sigma}}}{|z-y|^{n-2 \sigma}} \,\ud y.
\end{equation}
In the following we will adapt some arguments from Marques \cite{MA2005} for the Yamabe equation; see also Li-Zhang \cite{LZCompactnessII2005}, Niu-Peng-Xiong \cite{NPXCompactness2018}, and Li-Xiong \cite{LXCompactness2019}.

\begin{lemma}\label{lem:Expansions2}
Assume as in Lemma \ref{lem:Blowup3}. Suppose $\rho=1 .$ If $n \geq 4 \sigma$, we have for $|x| \leq m_{i}^{(p_{i}-1)/2 \sigma}$,
$$
|\Phi_{i}(x)-\Theta_{1}(x)| \leq C\left\{\begin{array}{ll}
m_{i}^{-2}(1+\|\nabla^{2} a_{i}\|_{L^{\infty}(B_{1})}) & \text { if }\, n<4 \sigma+2, \\
m_{i}^{-2}(1+\|\nabla^{2} a_{i}\|_{L^{\infty}(B_{1})} \ln m_{i}) & \text { if }\, n=4 \sigma+2, \\
m_{i}^{-2}(1+\|\nabla^{2} a_{i}\|_{L^{\infty}(B_{1})} m_{i}^{(2n-8\sigma-4)/(n-2 \sigma)}) & \text { if }\, n>4 \sigma+2,
\end{array}\right.
$$
where $\Phi_{i}(x)=m_{i}^{-1} u_{i}(m_{i}^{-(p_{i}-1)/2 \sigma} x+x_{i})$, $m_{i}=u_{i}(x_i)$, and $C>0$ depends only on $n, \sigma$ and $A_{0} .$
\end{lemma}

\begin{proof}
For brevity, let $\ell_{i}=m_{i}^{(p_{i}-1)/2 \sigma}$. By the equation satisfied by $u_{i}$, we have
\begin{equation}\label{eq:Phieq}
\Phi_{i}(x)=\int_{B_{\ell_i}} \frac{\tilde{a}_{i}(y) \Phi_{i}(y)+\Phi_{i}(y)^{p_{i}}}{|x-y|^{n-2 \sigma}} \,\ud y+\bar{b}_{i}(x),
\end{equation}
where
\begin{equation}\label{eq:tildeai}
\tilde{a}_{i}(y)=m_{i}^{1-p_{i}} a_{i}(\ell_{i}^{-1} y+x_{i}),
\end{equation}
\begin{equation}\label{eq:barhi}
\bar{b}_{i}(x)=m_{i}^{-1} b_{i}(\ell_{i}^{-1} x+x_{i}),
\end{equation}
and
$$
b_{i}(x)=\int_{B_{3} \backslash B_{1}(x_i)} \frac{a_{i}(y) u_{i}(y)+u_{i}(y)^{p_{i}}}{|x-y|^{n-2 \sigma}} \,\ud y+h_i(x).
$$
Since 0 is an isolated simple blow up point of $u_{i}$, it follows from Proposition \ref{pro:Blowup6} that
$$
u_{i}(x) \leq C m_{i}^{-1}|x-x_i|^{2\sigma-n} \quad \text { for all }\,|x-x_i|\leq1.
$$
Using the same argument as in the proof of \eqref{eq:binablabi}, we have $\bar{b}_{i}(x) \leq C m_{i}^{-2}$ for $x \in B_{\ell_{i}}$.

Let $x=\ell_{i} z$. By \eqref{eq:Thetalam} with $\lambda=\ell_{i}$ we have for $|x| \leq \ell_{i}$,
\begin{equation}\label{eq:Theta1eq}
\begin{aligned}
\Theta_{1}(x) &=\int_{B_{\ell_i}} \frac{\Theta_{1}(y)^{\frac{n+2 \sigma}{n-2 \sigma}}}{|x-y|^{n-2 \sigma}} \,\ud y+\ell_i^{\frac{2\sigma-n}{2}}\theta_{\ell_i}(\ell_i^{-1}x) \\
&=:\int_{B_{\ell_i}} \frac{\Theta_{1}(y)^{p_i}+T_i(y)}{|x-y|^{n-2 \sigma}} \,\ud y+\bar{\theta}_i(x),
\end{aligned}
\end{equation}
where
\begin{equation}\label{eq:Ti}
T_{i}(y)=\Theta_{1}(y)^{\frac{n+2\sigma}{n-2\sigma}}-\Theta_{1}(y)^{p_{i}},
\end{equation}
\begin{equation}\label{eq:barthetai}
\bar{\theta}_i(x)=\ell_{i}^{\frac{2 \sigma-n}{2}} \theta_{\ell_{i}}(\ell_{i}^{-1} x),
\end{equation}
and
$$
\theta_{\lambda}(z)=:\int_{B_1^c} \frac{\Theta_{\lambda}(y)^{\frac{n+2 \sigma}{n-2 \sigma}}}{|z-y|^{n-2 \sigma}} \,\ud y.
$$
Using the same argument as in the proof of \eqref{eq:binablabi} again, we have $\bar{\theta}_{i}(x)\leq C m_{i}^{-2}$ for $x \in B_{\ell_{i}}$.

Let
$$
\Lambda_{i}=\max _{|x| \leq \ell_{i}}|\Phi_{i}(x)-\Theta_{1}(x)|.
$$
By Proposition \ref{pro:Blowup6} and Lemma \ref{lem:Harnack}, we have for any $0<\varepsilon<1$ and $\varepsilon \ell_{i} \leq|x| \leq \ell_{i}$,
$$
|\Phi_{i}(x)-\Theta_{1}(x)| \leq |\Phi_{i}(x)|+|\Theta_{1}(x)|\leq C(\varepsilon) m_{i}^{-2},
$$
where we used $m_{i}^{\tau_{i}}=1+o(1)$. Hence, we may assume that $\Lambda_{i}$ is achieved at some point $|z_{i}| \leq \ell_{i}/2$, otherwise the proof is finished. Let
$$
V_{i}(x):=\frac{\Phi_{i}(x)-\Theta_{1}(x)}{\Lambda_{i}}.
$$
It follows from \eqref{eq:Phieq} and \eqref{eq:Theta1eq} that $V_{i}$ satisfies
\begin{equation}\label{eq:Vieq}
V_i(x)=\int_{B_{\ell_i}} \frac{c_i(y)V_i(y)+\Lambda_i^{-1}\tilde{a}_{i}(y) \Phi_{i}(y)-\Lambda_i^{-1}T_{i}(y)}{|x-y|^{n-2 \sigma}} \,\ud y+\frac{\bar{b}_{i}(x)-\bar{\theta}_{i}(x)}{\Lambda_i},
\end{equation}
where
\begin{equation}\label{eq:bi}
c_{i}(y)=\frac{\Phi_{i}(y)^{p_{i}}-\Theta_{1}(y)^{p_i}}{\Phi_{i}(y)-\Theta_{1}(y)}.
\end{equation}
By Taylor expansion of $a_{i}$ at $x_{i}$, we have
$$
a_{i}(\ell_{i}^{-1} y+x_{i}) \leq a_{i}(x_{i})+\ell_{i}^{-1}|y||\nabla a_{i}(x_{i})|+\ell_{i}^{-2}|y|^{2}\|\nabla^{2} a_{i}\|_{L^{\infty}(B_1)}.
$$
Since $\Phi_{i}(y) \leq C \Theta_{1}(y)$, it follows from Lemma \ref{lem:Expansions1} and Lemma \ref{lem:Blowup9} that
$$
\int_{B_{\ell_i}} \frac{\tilde{a}_{i}(y) \Phi_{i}(y)}{|x-y|^{n-2\sigma}} \,\ud y \leq C \alpha_{i}
$$
with
\begin{equation}\label{eq:alphai}
\alpha_{i}:=\left\{\begin{array}{ll}
m_{i}^{-2}(1+\|\nabla^{2} a_{i}\|_{L^{\infty}(B_{1})}) \quad& \text { if }\, 4\sigma\leq n<4 \sigma+2, \\
m_{i}^{-2}(1+\|\nabla^{2} a_{i}\|_{L^{\infty}(B_{1})} \ln m_{i}) \quad& \text { if }\, n=4 \sigma+2, \\
m_{i}^{-2}(1+\|\nabla^{2} a_{i}\|_{L^{\infty}(B_{1})} m_{i}^{(2n-8\sigma-4)/(n-2 \sigma)}) \quad& \text { if }\, n>4 \sigma+2,
\end{array}\right.
\end{equation}
and
\begin{equation}\label{eq:|bi|}
|c_{i}(y)|\leq C \Theta_{1}(y)^{p_{i}-1} \leq C(1+|y|)^{-3 \sigma}, \quad y \in B_{\ell_{i}}.
\end{equation}
Noticing that $\|V_{i}\|_{L^{\infty}(B_{\ell_{i}})} \leq 1$, we have
$$
\int_{B_{\ell_i}} \frac{|c_i(y)V_i(y)|}{|x-y|^{n-2 \sigma}} \,\ud y\leq C\int_{B_{\ell_i}} \frac{1}{|x-y|^{n-2 \sigma}(1+|y|)^{3\sigma}} \,\ud y\leq C(1+|x|)^{-\sigma}.
$$
Since
\begin{equation}\label{eq:|Ti|}
|T_{i}(y)|\leq C\tau_i|\log \Theta_1(y)|(1+|y|)^{-p_i(n-2\sigma)},
\end{equation}
we have
$$
\int_{B_{\ell_i}} \frac{|T_{i}(y)|}{|x-y|^{n-2 \sigma}} \,\ud y\leq C\tau_i\quad \text{ for }\, |x|\leq \ell_i/2.
$$
Hence, we obtain
\begin{equation}\label{eq:|Vi|}
|V_{i}(x)| \leq C(1+|x|)^{-\sigma}+\frac{C}{\Lambda_{i}}(\tau_i+\alpha_i+m_i^{-2})\quad\text{ for }\, |x|\leq\ell_i/2.
\end{equation}

If Lemma \ref{lem:Expansions2} were wrong, using \eqref{eq:tauiestimate} and the definition of $\alpha_i$, we have
$$
\|V_{i}\|_{L^{\infty}({B}_{\ell_{i} / 2})}\leq C(\ell_{i}^{-\sigma}+\frac{\tau_{i}+\alpha_{i}+m_{i}^{-2}}{\Lambda_{i}})\leq C(\ell_{i}^{-\sigma}+\frac{\alpha_{i}}{\Lambda_{i}}) \rightarrow 0
$$
as $i \rightarrow \infty$. By regularity theory in \cite{JLXThe2017}, $\{V_{i}(y)\}$ are locally uniformally bounded in $C^{2,\alpha}$ for some $0<\alpha<1 .$ By \eqref{eq:Vieq}, it follows from Arzelà-Ascoli theorem and Lebesgue dominated convergence theorem that, after passing to subsequence,
$$
V_{i}(x) \rightarrow v(x) \quad \text { in }\, C_{l o c}^{2}(\mathbb{R}^{n})
$$
for some $v \in C_{l o c}^{2}(\mathbb{R}^{n}) \cap L^{\infty}(\mathbb{R}^{n})$ satisfying
$$
v(x)=C(n, \sigma) \int_{\mathbb{R}^{n}} \frac{\Theta_{1}(y)^{\frac{4 \sigma}{n-2 \sigma}} v(y)}{|x-y|^{n-2 \sigma}} \,\ud y.
$$
It follows from the non-degeneracy result, see, e.g., Lemma 5.1 of \cite{LXCompactness2019}, that
$$
v(x)=c_{0}\Big(\frac{n-2 \sigma}{2} \Theta_{1}(x)+x \cdot\nabla \Theta_{1}(x)\Big)+\sum_{j=1}^{n} c_{j} \partial_{j} \Theta_{1}(x),
$$
where $c_{0}, \cdots, c_{n}$ are constants. By Proposition \ref{pro:Blowup1}, $v(0)=0$ and $\nabla v(0)=0$. Hence, $v\equiv0$. However, $v(z_{i})=1 .$ We obtain a contradiction, and the lemma follows.
\end{proof}

\begin{lemma}\label{lem:Expansions3}
Under the hypotheses in Lemma \ref{lem:Expansions2}, we have, for every $|x| \leq m_{i}^{(p_{i}-1)/2\sigma}$,
$$
\begin{aligned}
&|\Phi_{i}(x)-\Theta_{1}(x)| \\
\leq& C\left\{\begin{array}{ll}
\max \{m_{i}^{-2}, \|\nabla^{2} a_{i}\|_{L^{\infty}(B_{1})} m_{i}^{-2+{2 \sigma}/(n-2 \sigma)}(1+|x|)^{-\sigma}\} & \text { if }\, n=4 \sigma+2, \\
\max \{m_{i}^{-2}, \|\nabla^{2} a_{i}\|_{L^{\infty}(B_{1})} m_{i}^{-2+{(2n-8 \sigma-4)}/(n-2 \sigma)}(1+|x|)^{4 \sigma+2-n}\} & \text { if }\, n>4 \sigma+2.
\end{array}\right.
\end{aligned}
$$
\end{lemma}

\begin{proof}
Let $\alpha_{i}$ be defined in \eqref{eq:alphai}. We may assume $m_{i}^{-2}/\alpha_{i} \rightarrow 0$ as $i \rightarrow \infty$ for $n \geq 4 \sigma+2$; otherwise there exists a subsequence $i_{l}$ of $\{i\}$ such that $m_{i_{l}}^{-2} \geq C \alpha_{i_{l}}$ for some $C>0$ and the lemma follows from Lemma \ref{lem:Expansions2}. Let
$$
\alpha_{i}^{\prime}=\left\{\begin{array}{ll}
\|\nabla^{2} a_{i}\|_{L^{\infty}(B_{1})} m_{i}^{-2+{2 \sigma}/(n-2 \sigma)} \quad&\text { if }\, n=4 \sigma+2, \\
\|\nabla^{2} a_{i}\|_{L^{\infty}(B_{1})} m_{i}^{-2+{(2n-8\sigma-4)}/(n-2 \sigma)} \quad& \text { if }\, n>4 \sigma+2,
\end{array}\right.
$$
and
$$
V_{i}^{\prime}(x):=\frac{\Phi_{i}(x)-\Theta_{1}(x)}{\alpha_{i}^{\prime}}, \quad|x| \leq m_{i}^{(p_{i}-1)/2 \sigma}.
$$
Since ${m_{i}^{-2}}/{\alpha_{i}} \rightarrow 0$ and $\alpha_{i} \leq C\alpha_{i}^{\prime}$, it follows from Lemma \ref{lem:Expansions2} that $|V^{\prime}_{i}| \leq C$. Since $0<\Phi_{i} \leq C \Theta_{1}$ and $m_{i}^{\tau_{i}}=1+o(1)$, we have, for $\ell_{i}/2 \leq|x| \leq \ell_{i}$,
$$
|V^{\prime}_{i}(x)| \leq\left\{\begin{array}{ll}
C m_{i}^{-{2 \sigma}/(n-2 \sigma)} \quad& \text { if }\, n=4 \sigma+2, \\
C m_{i}^{-(2n-8\sigma-4)/(n-2 \sigma)} \quad& \text { if }\, n>4 \sigma+2,
\end{array}\right.
$$
and thus we only need to prove the proposition when $|x| \leq \ell_{i}/2$. Similar to \eqref{eq:Vieq}, $V^{\prime}_{i}(x)$ satisfies
$$
V^{\prime}_i(x)=\int_{B_{\ell_i}} \frac{c_i(y)V^{\prime}_i(y)+\alpha_{i}^{\prime-1}\tilde{a}_{i}(y) \Phi_{i}(y)-\alpha_{i}^{\prime-1}T_{i}(y)}{|x-y|^{n-2 \sigma}} \,\ud y+\frac{1}{\alpha_{i}^{\prime}}O(m_i^{-2}),
$$
where $c_{i}$, $\tilde{a}_{i}$, and $T_{i}$ are given by \eqref{eq:bi}, \eqref{eq:tildeai}, and \eqref{eq:Ti} respectively.

By Taylor expansion of $a_{i}$ at $x_{i}$, we have
$$
a_{i}(\ell_{i}^{-1} y+x_{i}) \leq a_{i}(x_{i})+\ell_{i}^{-1}|y||\nabla a_{i}(x_{i})|+\ell_{i}^{-2}|y|^{2}\| \nabla^{2} a_{i} \|_{L^{\infty}(B_{1})}.
$$
Since $\Phi_{i}(y) \leq C \Theta_{1}(y)$, by Lemma \ref{lem:Expansions1} and Lemma \ref{lem:Blowup9} we have: for $n=4 \sigma+2$,
$$
\begin{aligned}
\int_{B_{\ell_{i}}} \frac{\alpha_{i}^{\prime-1} \tilde{a}_{i}(y) \Phi_{i}(y)}{|x-y|^{n-2 \sigma}} \,\ud y & \leq C \int_{B_{\ell_{i}}} \frac{1}{|x-y|^{n-2 \sigma}(1+|y|)^{n-2 \sigma-2} m_{i}^{{2 \sigma}/(n-2 \sigma)}} \,\ud y \\
& \leq C \int_{B_{\ell_{i}}} \frac{1}{|x-y|^{n-2 \sigma}(1+|y|)^{n-\sigma-2}} \,\ud y \\
& \leq C(1+|x|)^{-\sigma};
\end{aligned}
$$
for $n>4 \sigma+2$,
$$
\begin{aligned}
\int_{B_{\ell_{i}}} \frac{\alpha_{i}^{\prime-1} \tilde{a}_{i}(y) \Phi_{i}(y)}{|x-y|^{n-2 \sigma}} \,\ud y & \leq C \int_{B_{\ell_{i}}} \frac{1}{|x-y|^{n-2 \sigma}(1+|y|)^{n-2 \sigma-2}} \,\ud y \\
& \leq C(1+|x|)^{4 \sigma+2-n}.
\end{aligned}
$$
Since $c_{i}(y) \leq C(1+|y|)^{-3 \sigma}$ and $|V^{\prime}_{i}(y)| \leq C$, it follows from Lemma \ref{lem:Expansions1} that
$$
\int_{B_{\ell_{i}}} \frac{|c_{i}(y) V^{\prime}_{i}(y)|}{|x-y|^{n-2\sigma}} \,\ud y \leq C(1+|x|)^{-\sigma}.
$$
Noticing that
$$
|T_{i}(y)|\leq C\tau_i|\log \Theta_1(y)|(1+|y|)^{-n-2\sigma},
$$
we have: for $n=4 \sigma+2$,
$$
\begin{aligned}
\int_{B_{\ell_{i}}} \frac{\alpha_{i}^{\prime-1} |T_{i}(y)|}{|x-y|^{n-2 \sigma}} \,\ud y\leq& C\int_{B_{\ell_{i}}}\frac{1}{|x-y|^{n-2 \sigma}(1+|y|)^{n+2\sigma}m_i^{2\sigma/(n-2\sigma)}}\,\ud y\\
\leq&C\int_{B_{\ell_{i}}}\frac{1}{|x-y|^{n-2 \sigma}(1+|y|)^{n+3\sigma}}\,\ud y\\
\leq&C(1+|x|)^{2\sigma-n};
\end{aligned}
$$
for $n>4 \sigma+2$,
$$
\begin{aligned}
\int_{B_{\ell_{i}}} \frac{\alpha_{i}^{\prime-1} |T_{i}(y)|}{|x-y|^{n-2 \sigma}} \,\ud y\leq& C\int_{B_{\ell_{i}}}\frac{1}{|x-y|^{n-2 \sigma}(1+|y|)^{n+2\sigma}}\,\ud y\\
\leq&C(1+|x|)^{2\sigma-n}.
\end{aligned}
$$
Thus
$$
|V^{\prime}_{i}(x)| \leq \left\{\begin{array}{ll}
C(1+|x|)^{-\sigma} \quad &\text { for }\, n=4\sigma+2, \\
C((1+|x|)^{-\sigma}+(1+|x|)^{4 \sigma+2-n}) \quad &\text { for }\, n >4\sigma+2.
\end{array}\right.
$$
If $4\sigma+2\leq n\leq 5\sigma+2$, the conclusion follows immediately from multiplying both sides of the above inequalities by $\alpha_{i}^{\prime}$. If $n > 5\sigma+2$, the above estimate gives $|V^{\prime}_{i}(x)| \leq C(1+|x|)^{-\sigma}$. Plugging this estimate to the term $\int_{B_{\ell_{i}}} c_{i}(y) V^{\prime}_{i}(y)|x-y|^{2\sigma-n} \,\ud y$ yields $|V^{\prime}_{i}(x)| \leq C(1+|x|)^{4\sigma+2-n}$ as long as $n \leq 6\sigma+2$. Repeating this process, we complete the proof.
\end{proof}

\begin{corollary}\label{cor:Expansions4}
Assume as Lemma \ref{lem:Expansions2}. We have
$$
\begin{aligned}
&|\nabla^{k}(\Phi_{i}(x)-\Theta_{1}(x))| \leq C(1+|x|)^{-k} \\
&\times\left\{\begin{array}{ll}
m_{i}^{-2} & \text { if }\, 4 \sigma \leq n<4 \sigma+2 \\
\max \{\|\nabla^{2} a_{i}\|_{L^{\infty}(B_{1})} m_{i}^{-2+{2 \sigma}/(n-2 \sigma)}(1+|x|)^{-\sigma}, m_{i}^{-2}\} & \text { if }\, n=4 \sigma+2, \\
\max \{\|\nabla^{2} a_{i}\|_{L^{\infty}(B_{1})} m_{i}^{-2+(2n-8\sigma-4)/(n-2 \sigma)}(1+|x|)^{4 \sigma+2-n}, m_{i}^{-2}\} & \text { if }\, n>4 \sigma+2,
\end{array}\right.
\end{aligned}
$$
for $k=0,1$, where $C>0$ depends only on $n, \sigma, A_{0}$ and $A_{1}$.
\end{corollary}

\begin{proof}
Considering the integral equation of $V^{\prime\prime}_{i}=\Phi_{i}-\Theta_{1}$, the conclusion follows immediately. Indeed, we can differentiate the integral equation for $V^{\prime\prime}_{i}$ directly.
\end{proof}

\section{Blow up local solutions of nonlinear integral equations}
\begin{proposition}\label{pro:Estimates1}
Assume as Lemma \ref{lem:Blowup3}. Assume further that $\|a_{i}\|_{C^{4}(B_{1})} \leq A_{0} .$ Then for $0<r<\rho$ there holds
$$
\begin{aligned}
&m_{i}^{2} \mathcal{P}_{\sigma}(x_i, r, u_i, p_i)=-m_{i}^{2} \mathcal{Q}_{\sigma}(x_{i}, r, u_{i}, p_{i}) \geq-C_{0} r^{-n} m_{i}^{-\frac{4 \sigma}{n-2 \sigma}}-C_{0}\|a_{i}\|_{L^{\infty}(B_{1})} r^{4 \sigma-n}\\
&+\frac{1}{C_{1}}\left\{\begin{array}{ll}
\sigma a_{i}(x_{i}) \ln (r m_{i}^{\frac{2}{n-2 \sigma}}) & n=4 \sigma, \\
\sigma a_{i}(x_{i}) m_{i}^{\frac{2(n-4 \sigma)}{n-2 \sigma}} & 4 \sigma<n<4 \sigma+2, \\
\sigma a_{i}(x_{i}) m_{i}^{\frac{2(n-4 \sigma)}{n-2 \sigma}}+\frac{\sigma+1}{2 n} \Delta a_{i}(x_{i}) \ln (r m_{i}^{\frac{2}{n-2 \sigma}}) & n=4 \sigma+2, \\
\sigma a_{i}(x_{i}) m_{i}^{\frac{2(n-4 \sigma)}{n-2 \sigma}}+\frac{\sigma+1}{2 n} \Delta a_{i}(x_{i}) m_{i}^{\frac{2(n-4 \sigma-2)}{n-2 \sigma}} & 4 \sigma+2<n<4 \sigma+4, \\
\beta_{i}-C_{0}\|a_{i}\|_{B_{1}} \ln (r m_{i}^{\frac{2}{n-2 \sigma}})& n=4\sigma+4,\\
\beta_i-C_{0}\|a_{i}\|_{B_{1}}m_{i}^{\frac{2(n-4 \sigma-4)}{n-2 \sigma}}& n>4\sigma+4,
\end{array}\right.
\end{aligned}
$$
where $\beta_{i}:=\sigma a_{i}(x_{i}) m_{i}^{\frac{2(n-4 \sigma)}{n-2 \sigma}}+\frac{\sigma+1}{2 n} \Delta a_{i}(x_{i}) m_{i}^{\frac{2(n-4 \sigma-2)}{n-2 \sigma}}$, $\|a_{i}\|_{B_{1}}:=\|a_{i}\|_{L^{\infty}(B_{1})} \times\|\nabla^{2} a_{i}\|_{L^{\infty}(B_{1})}+\|\nabla^{4} a_{i}\|_{L^{\infty}(B_{1})}$, $C_{0}>0$ depends only on $n, \sigma, A_{0}, A_{1}, \rho$ and independent of $r$ if $i$ is sufficiently large, and $C_{1}>0$ depends only on $n$ and $\sigma$.
\end{proposition}

\begin{proof}
By Proposition \ref{pro:Pohozaev}, we have
$$
\begin{aligned}
\mathcal{P}_{\sigma}(x_i, r, u_i, p_i)=&-\mathcal{Q}_{\sigma}(x_{i}, r, u_{i}, p_{i})\\
=&-\int_{B_{r}(x_{i})}((x-x_{i})_{k} \partial_{k} u_{i}+\frac{n-2 \sigma}{2} u_{i}) a_{i} u_{i} \,\ud x+\mathcal{N}(r, u_{i}),
\end{aligned}
$$
where
$$
\mathcal{N}(r, u_{i})=\frac{(n-2 \sigma) \tau_{i}}{2(p_{i}+1)} \int_{B_{r}(x_{i})} u_{i}^{p_{i}+1} \,\ud x-\frac{r}{p_{i}+1} \int_{\partial B_{r}(x_{i})} u_{i}^{p_{i}+1} \,\ud s.
$$
By Proposition \ref{pro:Blowup6},
$$
m_{i}^{2} \mathcal{N}(r, u_{i}) \geq-C r^{-n} m_{i}^{1-p_{i}}.
$$
By change of variables $y=\ell_{i}(x-x_{i})$ with $\ell_{i}=m_{i}^{(p_{i}-1)/2 \sigma}$, we have
$$
\begin{aligned}
\mathcal{E}_{i}(r): &=-m_{i}^{2} \int_{B_{r}(x_{i})}((x-x_{i})_{k} \partial_{k} u_{i}+\frac{n-2 \sigma}{2} u_{i}) a_{i} u_{i} \,\ud x\\
&=-m_{i}^{4-{n(p_{i}-1)}/2 \sigma} \int_{B_{\ell_{i} r}}(y_{k} \partial_{k} \Phi_{i}+\frac{n-2 \sigma}{2} \Phi_{i}) a_{i}(\ell_{i}^{-1} y+x_{i}) \Phi_{i} \,\ud y,
\end{aligned}
$$
where $\Phi_{i}(y)=m_{i}^{-1} u_{i}(\ell_{i}^{-1} y+x_{i}) .$

Let
$$
\hat{\mathcal{E}}_{i}(r):=-m_{i}^{4-{n(p_{i}-1)}/2 \sigma} \int_{B_{\ell_{i} r}}(y_{k} \partial_{k} \Theta_{1}+\frac{n-2 \sigma}{2} \Theta_{1}) a_{i}(\ell_{i}^{-1} y+x_{i}) \Theta_{1} \,\ud y.
$$
Making use of Corollary \ref{cor:Expansions4} and Lemma \ref{lem:Blowup4}, we have
$$
\begin{aligned}
&|\mathcal{E}_{i}(r)-\hat{\mathcal{E}}_{i}(r)| \\
\leq& C\|a_{i}\|_{L^{\infty}(B_{1})} m_{i}^{2-{4 \sigma}/(n-2 \sigma)} \int_{B_{\ell_{i} r}} \sum_{j=0}^{1}|\nabla^{j}(\Phi_{i}-\Theta_{1})|(1+|y|)^{2 \sigma-n+j} \,\ud y\\
\leq& C\|a_{i}\|_{L^{\infty}(B_{1})}\left\{\begin{array}{ll}
r^{2 \sigma} \quad& \text { if }\, n<4 \sigma+2, \\
\max \{r^{2 \sigma}, \|\nabla^{2} a_{i}\|_{L^{\infty}(B_{1})} r^{\sigma}\} \quad& \text { if }\, n=4 \sigma+2, \\
\max \{r^{2 \sigma}, \|\nabla^{2} a_{i}\|_{L^{\infty}(B_{1})} r^{6 \sigma+2-n}\} \quad& \text { if }\, 4 \sigma+2<n<6 \sigma+2, \\
\max \{r^{2 \sigma}, \|\nabla^{2} a_{i}\|_{L^{\infty}(B_{1})} \ln (r m_{i}^{{2}/(n-2 \sigma)})\} \quad& \text { if }\, n=6 \sigma+2, \\
\max \{r^{2 \sigma}, \|\nabla^{2} a_{i}\|_{L^{\infty}(B_{1})} m_{i}^{{2(n-6 \sigma-2)}/(n-2 \sigma)}\} \quad& \text { if }\, n>6 \sigma+2,
\end{array}\right.
\end{aligned}
$$
where $C>0$ depends only on $n, \sigma, A_{0}$ and $A_{1}$. Next, by the divergence theorem and direction computations we see that
$$
\begin{aligned}
\hat{\mathcal{E}}_{i}(r) =&-m_{i}^{2} \int_{B_{r}}(z_{k} \partial_{k} \Theta_{\ell_{i}}+\frac{n-2 \sigma}{2} \Theta_{\ell_{i}}) a_{i}(z+x_{i}) \Theta_{\ell_{i}} \,\ud z \\
\geq& m_{i}^{2} \int_{B_{r}}(\frac{1}{2} z_{k} \partial_{k} a_{i}(z+x_{i})+\sigma a_{i}(z+x_{i})) \Theta_{\ell_{i}}^{2} \,\ud z-C\|a_{i}\|_{L^{\infty}(B_{1})} r^{4 \sigma-n}\\
\geq& m_{i}^{2} \int_{B_{r}}(\sigma a_{i}(x_{i})+(\sigma+\frac{1}{2}) z_{k} \partial_{k} a_{i}(x_{i})+(\frac{1}{2}+\frac{\sigma}{2}) \partial_{k l} a_{i}(x_{i}) z_{k} z_{l}\\
&+(\frac{1}{4}+\frac{\sigma}{6}) \partial_{j k l} a_{i}(x_{i}) z_{j} z_{k} z_{l}) \Theta_{\ell_{i}}^{2} \,\ud z\\
&-C m_{i}^{2}\|\nabla^{4} a_{i}\|_{L^{\infty}(B_{1})} \int_{B_{r}}|z|^{4} \Theta_{\ell_{i}}^{2} \,\ud z-C\|a_{i}\|_{L^{\infty}(B_{1})} r^{4 \sigma-n}\\
=&m_{i}^{2} \int_{B_{r}}(\sigma a_{i}(x_{i})+\frac{\sigma+1}{2 n} \Delta a_{i}(x_{i})|z|^{2}) \Theta_{\ell_{i}}^{2} \,\ud z\\
&-C m_{i}^{2}\|\nabla^{4} a_{i}\|_{L^{\infty}(B_{1})} \int_{B_{r}}|z|^{4} \Theta_{\ell_{i}}^{2} \,\ud z-C\|a_{i}\|_{L^{\infty}(B_{1})} r^{4 \sigma-n},
\end{aligned}
$$
$$
\begin{aligned}
&m_{i}^{2} \int_{B_{r}}(\sigma a_{i}(x_{i})+\frac{\sigma+1}{2 n} \Delta a_{i}(x_{i})|z|^{2}) \Theta_{\ell_{i}}^{2} \,\ud z\\
\geq&C\left\{\begin{array}{ll}
\sigma a_{i}(x_{i}) \ln (r m_{i}^{\frac{2}{n-2 \sigma}}) \quad&\, n=4 \sigma, \\
\sigma a_{i}(x_{i}) m_{i}^{\frac{2(n-4 \sigma)}{n-2 \sigma}} \quad&\, 4 \sigma<n<4 \sigma+2, \\
\sigma a_{i}(x_{i}) m_{i}^{\frac{2(n-4 \sigma)}{n-2 \sigma}}+\frac{\sigma+1}{2 n} \Delta a_{i}(x_{i}) \ln (r m_{i}^{\frac{2}{n-2 \sigma}}) \quad&\, n=4 \sigma+2, \\
\sigma a_{i}(x_{i}) m_{i}^{\frac{2(n-4 \sigma)}{n-2 \sigma}}+\frac{\sigma+1}{2 n} \Delta a_{i}(x_{i}) m_{i}^{\frac{2(n-4 \sigma-2)}{n-2 \sigma}}\quad &\, n>4 \sigma+2,
\end{array}\right.
\end{aligned}
$$
and
$$
m_{i}^{2} \int_{B_{r}}|z|^{4} \Theta_{\ell_{i}}^{2} \,\ud z \leq C\left\{\begin{array}{ll}
r^{4 \sigma+4-n} \quad&\, n<4 \sigma+4, \\
\ln (r m_{i}^{\frac{2}{n-2 \sigma}}) \quad&\, n=4 \sigma+4, \\
m_{i}^{\frac{2(n-4 \sigma-4)}{n-2 \sigma}} \quad&\, n>4 \sigma+4,
\end{array}\right.
$$
where $C>0$ depends only on $n, \sigma$ and $\sup _{i}\|\nabla^{4} a_{i}\|_{L^{\infty}(B_{1})} .$ Since $4 \sigma+4<6 \sigma+2$ and
$$
-m_{i}^{2}\mathcal{Q}_{\sigma}(x_{i}, r, u_{i}, p_{i}) \geq \hat{\mathcal{E}}_{i}(r)-|\mathcal{E}_{i}(r)-\hat{\mathcal{E}}_{i}(r)|+m_{i}^{2} \mathcal{N}(r, u_{i}),
$$
the proposition follows immediately.
\end{proof}

\begin{proposition}\label{pro:Estimates2}
Assume as Lemma \ref{lem:Harnack}. Suppose that for large $i$,
\begin{enumerate}[(i)]
  \item $\beta_{i} \geq 0$ if $4 \sigma+2 \leq n<4 \sigma+4$;
  \item $\beta_{i} \geq(C_{0}+1)\|a_{i}\|_{B_{1}} \ln m_{i}$ if $n=4\sigma+4$;
  \item $\beta_{i} \geq(C_{0}+1)\|a_{i}\|_{B_{1}} m_{i}^{\frac{2(n-4 \sigma-4)}{n-2 \sigma}}$ if $n>4\sigma+4$;
\end{enumerate}
where
$$
\beta_{i}:=\begin{cases}
\sigma a_{i}(x_{i}) m_{i}^{\frac{2(n-4 \sigma)}{n-2 \sigma}}+\frac{\sigma+1}{2 n} \Delta a_{i}(x_{i}) \ln m_{i}\, & \text { if }\, n=4 \sigma+2, \\ \sigma a_{i}(x_{i}) m_{i}^{\frac{2(n-4 \sigma)}{n-2 \sigma}}+\frac{\sigma+1}{2 n} \Delta a_{i}(x_{i}) m_{i}^{\frac{2(n-4 \sigma-2)}{n-2 \sigma}}\, & \text { if }\, n>4 \sigma+2,
\end{cases}
$$
$m_{i}=u_{i}(x_{i})$, and $C_{0}$ is the constant in Proposition \ref{pro:Estimates1} with $\rho=1$, then, after passing to a subsequence, $x_{i} \rightarrow 0$ is an isolated simple blow up point of $\{u_{i}\}$.
\end{proposition}

\begin{proof}
By Proposition \ref{pro:Blowup1}, $r^{2 \sigma /(p_{i}-1)} \bar{u}_{i}(r)$ has precisely one critical point in the interval $0<r<r_{i}:=R_{i} u_{i}(x_{i})^{-(p_{i}-1)/2 \sigma}$, where $\bar{u}_{i}(r)=|\partial B_r(x_i)|^{-1}\int_{\partial B_r(x_i)} u_i$. If the proposition were wrong, let $\mu_{i} \geq r_{i}$ be the second critical point of $r^{2 \sigma /(p_{i}-1)} \bar{u}_{i}(r)$. Then there must hold
$$
\lim _{i \rightarrow \infty} \mu_{i}=0.
$$
Without loss of generality, we assume that $x_{i}=0 .$ Define
$$
\phi_{i}(x):=\mu_{i}^{2 \sigma /(p_{i}-1)} u_{i}(\mu_{i} x), \quad x \in \mathbb{R}^{n}.
$$
Clearly, $\phi_{i}$ satisfies
$$
\phi_i(x)=\int_{{B_{3/\mu_i}}} \frac{\tilde{a}_{i}(y) \phi_i(y)+\phi_i(y)^{p_i}}{| x- y|^{n-2 \sigma}} \,\ud y+\tilde{h}_i(x),
$$
$$
|x|^{2 \sigma /(p_{i}-1)} \phi_{i}(x) \leq A_1 \quad \text{ for }\, |x|<1 / \mu_{i},
$$
$$
\lim _{i \rightarrow \infty} \phi_{i}(0)=\infty,
$$
$$
r^{2 \sigma /(p_{i}-1)} \bar{\phi}_{i}(r) \text { has precisely one critical point in } 0<r<1,
$$
and
$$
\frac{\ud}{\ud r}\{r^{2 \sigma /(p_{i}-1)} \bar{\phi}_{i}(r)\}\Big|_{r=1}=0,
$$
where $\tilde{a}_{i}(y)=\mu_{i}^{2 \sigma} a_{i}(\mu_{i} y)$, $\tilde{h}_i(x)=\mu_i^{2\sigma/(p_i-1)}h_i(\mu_ix)$ and $\bar{\phi}_{i}(r)=|\partial B_{r}|^{-1} \int_{\partial B_{r}} \phi_{i}$. Therefore, 0 is an isolated simple blow up point of $\phi_{i}$.

We claim that
\begin{equation}\label{eq:Esti2-1}
\phi_{i}(0) \phi_{i}(x) \rightarrow \frac{b}{|x|^{n-2 \sigma}}+b \quad \text { in }  C_{l o c}^{2}(\mathbb{R}^{n} \backslash\{0\}),
\end{equation}
where $b>0$ is some constant.

By the equation of $\phi_{i}$, we have for all $|x| \leq 1 / \mu_{i}$,
$$
\begin{aligned}
\phi_{i}(0) \phi_{i}(x) &= \int_{B_{3/\mu_i}} \frac{\tilde{a}_{i}(y)\phi_{i}(0)\phi_i(y)+\phi_{i}(0)^{1-p_i}(\phi_{i}(0)\phi_i(y))^{p_i}}{|x-y|^{n-2 \sigma}} \,\ud y+\phi_{i}(0)\tilde{h}_i(x) \\
&=\int_{B_{t}}+\int_{{B_{3/\mu_i}} \backslash B_{t}} \frac{\tilde{a}_{i}(y)\phi_{i}(0)\phi_i(y)+\phi_{i}(0)^{1-p_i}(\phi_{i}(0)\phi_i(y))^{p_i}}{|x-y|^{n-2 \sigma}} \,\ud y+\phi_{i}(0)\tilde{h}_i(x),
\end{aligned}
$$
where $t>1$ is an arbitrarily fixed constant.

By the same proof of Proposition \ref{pro:Blowup6} we have, up to a subsequence,
\be\label{eq:Bt}
\int_{B_{t}} \frac{\tilde{a}_{i}(y)\phi_{i}(0)\phi_i(y)+\phi_{i}(0)^{1-p_i}(\phi_{i}(0)\phi_i(y))^{p_i}}{|x-y|^{n-2 \sigma}} \,\ud y \rightarrow \frac{b}{|x|^{n-2 \sigma}} \quad \text { in }\,  C_{l o c}^{2}(B_{t} \backslash\{0\}).
\ee
For any fixed large $R>t+1$, it follows from \eqref{eq:Bt} that
$$
\int_{t \leq|y| \leq R} \frac{\tilde{a}_{i}(y)\phi_{i}(0)\phi_i(y)+\phi_{i}(0)^{1-p_i}(\phi_{i}(0)\phi_i(y))^{p_i}}{|x-y|^{n-2 \sigma}} \,\ud y \rightarrow 0 \quad \text { for }\, x \in B_{t}
$$
as $i \rightarrow \infty$, since the constant $b$ is independent of $t$.

Notice that for any $x \in B_{t}$,
$$
Q_{i}(x):=\int_{B_{3 / \mu_{i}} \backslash B_{R}} \frac{\tilde{a}_{i}(y)\phi_{i}(0)\phi_i(y)+\phi_{i}(0)^{1-p_i}(\phi_{i}(0)\phi_i(y))^{p_i}}{|x-y|^{n-2 \sigma}} \,\ud y \leq C(n, \sigma, A_0) \phi_{i}(0)\max _{\partial B_{R}} \phi_{i}.
$$
Since $\max _{\partial B_{R}} \phi_{i} \leq C R^{2\sigma-n} \phi_{i}(0)^{-1}$, it follows from the proof of \eqref{eq:varitovar} that, after passing to a subsequence,
$$
Q_{i}(x) \rightarrow q(x) \quad \text { in }\, C_{l o c}^{2}(B_{t}) \, \text { as }\, i \rightarrow \infty
$$
for some $q \in C^{2}(B_{t})$. For any $x \in B_{t}$ and $|y|>R$, we have
$$
|\nabla_x |x-y|^{2 \sigma-n}|\leq C(n,\sigma)|x-y|^{2\sigma-n-1}\leq \frac{C(n,\sigma)}{R-t}|x-y|^{2 \sigma-n}.
$$
Therefore, we have $|\nabla q(x)| \leq \frac{C(n,\sigma)}{R-t} q(x)$. By sending $R \rightarrow \infty$, we have $|\nabla q(x)| \equiv 0$ for any $x \in B_{t} .$ Thus,
$$
q(x) \equiv q(0) \quad \text { for all }\, x \in B_{t}.
$$

By Proposition \ref{pro:Blowup6} we have $\tilde{h}_{i}(e) \leq \phi_{i}(e) \leq C \phi_{i}(0)^{-1}$ for any $e \in \mathbb{S}^{n}$, where $C>0$ is independent of $i$. It follows from the assumption \eqref{eq:Harnackhi} on $h_{i}$ that
$$
v_{i}(0) \tilde{h}_{i}(x) \leq C\quad \text { for all }\, |x| \leq 2 / \mu_{i}
$$
and
$$
\|\nabla(\phi_{i}(0) \tilde{h}_{i})\|_{L^{\infty}(B_{1/\mu_{i}})} \leq \mu_{i}\|\phi_{i}(0) \tilde{h}_{i}\|_{L^{\infty}(B_{1/\mu_{i}})} \leq C \mu_{i}.
$$
Hence, for some constant $c_{0} \geq 0$, we have, along a subsequence,
$$
\lim _{i \rightarrow \infty}\|\phi_{i}(0) \tilde{h}_{i}(\cdot)-c_{0}\|_{L^{\infty}(B_{t})}=0, \quad \forall\, t>0.
$$

Since
$$
\frac{\ud}{\ud r}\{r^{2 \sigma /(p_{i}-1)} \phi_{i}(0) \bar{\phi}_{i}(r)\}\Big|_{r=1}=\phi_{i}(0) \frac{\ud}{\ud r}\{r^{2 \sigma /(p_{i}-1)} \bar{\phi}_{i}(r)\}\Big|_{r=1}=0,
$$
we have, by sending $i$ to $\infty,$ that
$$
q(0)+c_0=b>0.
$$
Therefore, \eqref{eq:Esti2-1} follows.

Define
$$
b_{i}(x):=\int_{{B_{3/\mu_i}} \backslash B_{\delta}} \frac{\tilde{a}_{i}(y)\phi_i(y)+\phi_i(y)^{p_i}}{|x-y|^{n-2 \sigma}} \,\ud y+\tilde{h}_i(x),
$$
$$
c_i(x):=\int_{{B_{3/\mu_i}} \backslash B_{t}} \frac{\tilde{a}_{i}(y)\phi_{i}(0)\phi_i(y)+\phi_{i}(0)^{1-p_i}(\phi_{i}(0)\phi_i(y))^{p_i}}{|x-y|^{n-2 \sigma}} \,\ud y+\phi_{i}(0)\tilde{h}_i(x),
$$
and
$$
d_i(x):=\int_{{B_{3/\mu_i}} \backslash B_{t}} \frac{\tilde{a}_{i}(y)\phi_{i}(0)\phi_i(y)+\phi_{i}(0)^{1-p_i}(\phi_{i}(0)\phi_i(y))^{p_i}}{|x-y|^{n-2 \sigma}} \,\ud y,
$$
then
$$
\phi_{i}(0) b_{i}(x) \geq c_{i}(x) \geq 2^{2 \sigma-n} d_{i}(0)+\phi_{i}(0)\tilde{h}_i(x)\geq 2^{2 \sigma-n} (d_{i}(0)+\phi_{i}(0)\tilde{h}_i(x)) \rightarrow 2^{2 \sigma-n} b
$$
for $x \in B_{\delta}$ provided that $\delta$ is small, and
$$
|\nabla b_{i}(x)| \leq C\textcolor{red}{\phi_i(0)^{-1}}.
$$
Hence, it follows from Proposition \ref{pro:Blowup1}, Lemma \ref{lem:Blowup4} and Proposition \ref{pro:Blowup6} that
$$
\begin{aligned}
\int_{B_{\delta}} (\tilde{a}_i\phi_i+\phi_i^{p_i})b_i& \geq C^{-1} b \phi_{i}(0)^{-1} \int_{B_{{\phi_{i}(0)}^{-(p_{i}-1) / 2 \sigma}}} \phi_{i}^{p_{i}} \\
& \geq C^{-1} \phi_{i}(0)^{-2} \int_{B_{1}}(1+\bar{c}|x|^{2})^{(2 \sigma-n) / 2},
\end{aligned}
$$
and
$$
\Big|\int_{B_{\delta}} x \nabla b_{i}(\tilde{a}_i\phi_i+\phi_i^{p_i})\Big| \leq C \phi_{i}(0)^{-1} \int_{B_{\delta}}|x| (\tilde{a}_i\phi_i+\phi_i^{p_i})=o(1) \phi_{i}(0)^{-2},
$$
Therefore,
\begin{equation}\label{eq:Esti2-2}
\lim _{i \rightarrow \infty} \phi_{i}(0)^{2}\Big(-\frac{n-2 \sigma}{2} \int_{B_{\delta}} (\tilde{a}_i\phi_i+\phi_i^{p_i})b_i-\int_{B_{\delta}} x \nabla b_{i}(\tilde{a}_i\phi_i+\phi_i^{p_i})\Big)<0.
\end{equation}

If $4\sigma+2\leq n<4 \sigma+4$, by Proposition \ref{pro:Estimates1} and item (i) in the assumptions we have
$$
\liminf _{i \rightarrow \infty} \phi_{i}(0)^{2} \mathcal{P}_{\sigma}(0, \delta, \phi_{i}, p_i) \geq 0.
$$
This contradicts to \eqref{eq:Esti2-2}. Hence $x_{i} \rightarrow 0$ has to be an isolated simple blow up point of $\{u_{i}\}$ upon passing to a subsequence.

If $n \geq 4 \sigma+4$, let
$$
\begin{aligned}
\tilde{\beta}_{i}: &=\sigma \tilde{a}_{i}(x_{i}) \phi_{i}(0)^{\frac{2(n-4 \sigma)}{n-2 \sigma}}+\frac{\sigma+1}{2 n} \Delta \tilde{a}_{i}(x_{i}) \phi_{i}(0)^{\frac{2(n-4 \sigma-2)}{n-2 \sigma}} \\
&=(1+o(1)) \mu_{i}^{n-2 \sigma} \beta_{i}.
\end{aligned}
$$
Since $\|\tilde{a}_{i}\|_{C^{4}(B_{1})} \leq \mu_{i}^{2 \sigma} A_{0}$ and $\|\tilde{a}_{i}\|_{B_{1}} \leq \mu_{i}^{2 \sigma+4}\|a_{i}\|_{B_{1}}$, we see that
$$
\begin{aligned}
&\tilde{\beta}_{i}-C_{0}\|\tilde{a}_{i}\|_{B_{1}} \ln \phi_{i}(0)-C_{0}\|\tilde{a}_{i}\|_{L^{\infty}(B_{1})} \delta^{4 \sigma-n} \\
\geq &(1+o(1)) \mu_{i}^{2 \sigma+4} \beta_{i}-C_{0} \mu_{i}^{2 \sigma+4}\|a_{i}\|_{B_{1}} \ln m_{i}-C_{0}\mu_{i}^{2 \sigma+4}\|{a}_{i}\|_{B_{1}} \ln \mu_{i}^{\frac{2 \sigma}{p_{i}-1}} \\
&-C_{0} \mu_{i}^{2 \sigma}\|a_{i}\|_{L^{\infty}(B_{1})} \delta^{-4} \\
\geq &(1+o(1)) \mu_{i}^{2\sigma+4}(C_{0}+1)\|a_{i}\|_{B_{1}} \ln m_{i}-C_{0} \mu_{i}^{2 \sigma+4}\|a_{i}\|_{B_{1}} \ln m_{i}-C_{0}\mu_{i}^{2 \sigma+4}\|{a}_{i}\|_{B_{1}} \ln \mu_{i}^{\frac{2 \sigma}{p_{i}-1}} \\
&-C_{0} \mu_{i}^{2 \sigma}\|a_{i}\|_{L^{\infty}(B_{1})} \delta^{-4} \geq 0
\end{aligned}
$$
for $n=4\sigma+4$ and
$$
\begin{aligned}
&\tilde{\beta}_{i}-C_{0}\|\tilde{a}_{i}\|_{B_{1}} \phi_{i}(0)^{\frac{2(n-4 \sigma-4)}{n-2 \sigma}}-C_{0}\|\tilde{a}_{i}\|_{L^{\infty}(B_{1})} \delta^{4 \sigma-n} \\
\geq&(1+o(1)) \mu_{i}^{n-2 \sigma} \beta_{i}-C_{0}(1+o(1)) \mu_{i}^{n-2 \sigma}\|a_{i}\|_{B_{1}} m_{i}^{\frac{2(n-4 \sigma-4)}{n-2 \sigma}}-C_{0} \mu_{i}^{2 \sigma}\|a_{i}\|_{L^{\infty}(B_{1})} \delta^{4 \sigma-n} \\
\geq&(1+o(1)) \mu_{i}^{n-2 \sigma}(C_{0}+1)\|a_{i}\|_{B_{1}} m_{i}^{\frac{2(n-4 \sigma-4)}{n-2 \sigma}}-C_{0}(1+o(1)) \mu_{i}^{n-2 \sigma}\|a_{i}\|_{B_{1}} m_{i}^{\frac{2(n-4 \sigma-4)}{n-2 \sigma}} \\
&-C_{0} \mu_{i}^{2 \sigma}\|a_{i}\|_{L^{\infty}(B_{1})} \delta^{4 \sigma-n} \geq 0
\end{aligned}
$$
for $n>4\sigma+4$. By Proposition \ref{pro:Estimates1},
$$
\liminf _{i \rightarrow \infty} \phi_{i}(0)^{2} \mathcal{P}_{\sigma}(0, \delta, \phi_{i}, p_i) \geq 0.
$$
This contradicts to \eqref{eq:Esti2-2}. Hence $x_{i} \rightarrow 0$ has to be an isolated simple blow up point of $\{u_{i}\}$ upon passing to a subsequence. Therefore, we complete the proof of Proposition \ref{pro:Estimates2}.
\end{proof}

\section{Proof of the main theorems}
Suppose that $1 \leq \sigma<n / 2$. Let $u \in C^{2}(B_3)$ be a positive solution of
\be\label{63}
u(x)=\int_{B_3} \frac{a(y) u(y)+u(y)^p}{|x-y|^{n-2\sigma}}\,\ud y+h(x) \quad \text { in }\, B_3,
\ee
where $0\leq a\in C^{2}(B_3)$, $1<p \leq \frac{n+2 \sigma}{n-2 \sigma}$ and $0\leq h\in C^{\infty}(B_3)$ satisfy \eqref{eq:hsatisfy}.

\begin{proposition}\label{pro:Proof1}
Assume as above. Then for any $0<\varepsilon<1$ and $R>1$, there exist large positive constants $C_{1}$ and $C_{2}$ depending only on $n, \sigma,\|a\|_{C^{2}(B_2)}, \varepsilon$ and $R$ such that the following statement holds. If
$$
\max _{\overline{B}_{2}} \operatorname{dist}(x, \partial B_{2})^{\frac{n-2 \sigma}{2}} u(x) \geq C_{1},
$$
then $p \geq \frac{n+2 \sigma}{n-2 \sigma}-\varepsilon$ and there exists a finite set $S$ of local maximum points of $u$ in $B_2$ such that:
\begin{enumerate}[(i)]
  \item For any $x \in S$, it holds
  $$
\|u(x)^{-1} u(u(x)^{-(p-1)/2 \sigma}\cdot+x)-(1+\bar{c}|\cdot|^{2})^{(2 \sigma-n)/{2}}\|_{C^{2}(B_{2 R})}<\varepsilon,
$$
  where $\bar{c}>0$ depends only on $n, \sigma$.
  \item If $x_{1}, x_{2} \in S$ and $x_{1} \neq x_{2}$, then
  $$
B_{R u(x_{1})^{-(p-1) / 2 \sigma}}(x_{1}) \cap B_{R u(x_{2})^{-(p-1) / 2 \sigma}}(x_{2})=\emptyset.
$$
  \item $u(x) \leq C_{2} \operatorname{dist}(x, S)^{-2 \sigma /(p-1)}$ for all $x \in B_2$.
\end{enumerate}
\end{proposition}

\begin{proof}
The proof of Proposition \ref{pro:Proof1} is similar to that of Proposition 4.1 in \cite{LPrescribing1995}, Lemma 3.1 in \cite{SZPrescribed1996} and Proposition 5.1 in \cite{LZYamabe1999}, see also Proposition 3.1 in \cite{JLXThe2017}. We omit it here.
\end{proof}

\begin{proposition}\label{pro:Proof2}
Let $u \in C^{2}(B_3)$ be a solution of \eqref{63} with $0 \leq a \in C^{4}(B_3)$. Suppose that $\Delta a \geq 0$ on $\{x: a(x)<d\} \cap B_{2}$ for some constant $d>0$, and further that $\Delta a>\gamma>0$ on $\{x: a(x)<d\} \cap B_{2}$ for some constant $\gamma$ if $n \geq 4 \sigma+4 .$ Then for $\varepsilon>0$, $R>1$ and any solution of \eqref{63} with $\max _{\overline{B}_{2}} \operatorname{dist}(x, \partial B_{2})^{(n-2 \sigma)/{2}} u(x) \geq C_{1}$, we have
$$
|x_{1}-x_{2}| \geq \delta^{*}>0 \quad \text { for any } x_{1}, x_{2} \in S \cap B_{3 / 2} \text{ and } x_1\neq x_2,
$$
where $\delta^{*}$ depends only on $n, \sigma, d, \gamma, \varepsilon, R$ and $\|a\|_{C^{4}(B_{3})} .$
\end{proposition}

\begin{proof}
Suppose the contrary, for some $\varepsilon, R$ and $d>0$, there exist sequence $\{p_{i}\}$, nonnegative potentials $a_{i} \rightarrow a$ in $C^{4}(B_{3})$ with $\|a_{i}\|_{C^{4}(B_{3})} \leq A_{0}$, satisfying the assumptions for $a$, and a sequence of corresponding solutions $\{u_{i}\}_{i=1}^{\infty}$ such that
$$
\lim _{i \rightarrow \infty} \min _{j \neq l}|z_{i, j}-z_{i, l}|=0,
$$
where $z_{i, j}, z_{i, l} \in S_{i} \cap B_{3 / 2}$ associated to $u_{i}$ defined in Proposition \ref{pro:Proof1}.

Upon passing to a subsequence, we assume $z_{i, j}, z_{i, l} \rightarrow \bar{z} \in \overline{B}_{3 / 2}$. Define $f_{i}(z): S_{i} \rightarrow(0, \infty)$ by $f_{i}(z)=\min _{y \in S_{i} \backslash\{z\}}|z-y| .$ Let $R_{i} \rightarrow \infty$ with $R_{i} f_{i}(z_{i, j}) \rightarrow 0$ as $i \rightarrow \infty .$ By Lemma 2.1 in \cite{NPXCompactness2018}, we can find  $z_{i, 1} \in S_{i} \cap B_{2 R_{i} f_{i}(z_{i, j})}(z_{i, j})$ satisfying
$$
f_{i}(z_{i, 1}) \leq(2 R_{i}+1) f_{i}(z_{i, j}) \quad \text { and } \quad \min _{z \in S_{i} \cap B_{R_{i} f_{i}(z_{i, 1})}(z_{i, 1})} f_{i}(z) \geq \frac{1}{2} f_{i}(z_{i, 1}).
$$
Let $|z_{i, 2}-z_{i, 1}|=f_i(z_{i, 1})$. Let
$$
\Phi_{i}(x)=f_{i}(z_{i, 1})^{2 \sigma /(p_{i}-1)} u_{i}(f_{i}(z_{i, 1}) x+z_{i, 1}).
$$
The rest of the proof is divided into three steps:
\begin{enumerate}
  \item Prove that 0 and $x_{i}:=f_{i}(z_{i, 1})^{-1}(z_{i, 2}-z_{i, 1}) \rightarrow \bar{x}$ with $|\bar{x}|=1$ are two isolated blow up points of $\{\Phi_{i}(x)\}$.
  \item By Proposition \ref{pro:Estimates2}, after passing to a subsequence 0 and $x_{i} \rightarrow \bar{x}$ have to be isolated simple blow up points of $\{\Phi_{i}(x)\}$.
  \item Since $\Phi_{i}(0) \Phi_{i}(x)$ tends to a function with at least two poles, we can drive a contradiction by Pohozaev identity.
\end{enumerate}

For step 1 and 3, see the proof of Proposition 3.2 of \cite{JLXThe2017}. For step 2, we let
$$
\tilde{a}_{i}(x):=f_{i}(z_{i, 1})^{2 \sigma} a_{i}(f_{i}(z_{i, 1}) x+z_{i, 1})
$$
and verify assumptions in Proposition \ref{pro:Estimates2}. We only show it if $n \geq 4 \sigma+4$. By the assumption of $a_{i}$, we have
$$
\begin{aligned}
&\frac{\sigma \tilde{a}_{i}(0) \Phi_{i}(0)^{\frac{2(n-4 \sigma)}{n-2 \sigma}}+\frac{\sigma+1}{2 n} \Delta \tilde{a}_{i}(0) \Phi_{i}(0)^{\frac{2(n-4 \sigma-2)}{n-2 \sigma}}}{\|\tilde{a}_{i}\|_{B_{1 / 2}}} \\
\geq& \Phi_{i}(0)^{\frac{2(n-4 \sigma-2)}{n-2 \sigma}}\|a_{i}\|_{B_{2}}^{-1}\{\sigma a_{i}(z_{i, 1}) \Phi_{i}(0)^{\frac{4}{n-2 \sigma}} f_{i}(z_{i, 1})^{-2}+\frac{\sigma+1}{2 n} \Delta a_{i}(z_{i, 1})\} \\
\geq& \Phi_{i}(0)^{\frac{2(n-4 \sigma-2)}{n-2 \sigma}}\|a_{i}\|_{B_{2}}^{-1} \frac{\sigma+1}{2 n} \cdot \gamma\quad \text { for large }\, i.
\end{aligned}
$$

Since
$$
\frac{\Phi_{i}(0)^{\frac{2(n-4 \sigma-2)}{n-2 \sigma}}\|a_{i}\|_{B_{2}}^{-1} \frac{\sigma+1}{2 n} \cdot \gamma}{\ln \Phi_{i}(0)} \rightarrow \infty \quad \text { if }\, n=4 \sigma+4,
$$
and
$$
\frac{\Phi_{i}(0)^{\frac{2(n-4 \sigma-2)}{n-2 \sigma}}\|a_{i}\|_{B_{2}}^{-1} \frac{\sigma+1}{2 n} \cdot \gamma}{\Phi_{i}(0)^{\frac{2(n-4 \sigma-4)}{n-2 \sigma}}} \rightarrow \infty \quad \text { if }\, n>4\sigma+4,
$$
by Proposition \ref{pro:Estimates2}, 0 is an isolated simple blow up point of $\{\Phi_{i}(x)\}$. Similarly, one can show $x_{i} \rightarrow \bar{x}$ is an isolated simple blow up point of $\{\Phi_{i}(x)\}$. Therefore, we complete the proof of Proposition \ref{pro:Proof2}.
\end{proof}

{\noindent\bf Proof of Theorem \ref{thm:1.2}.}
\begin{proof}
We first prove that $\|u\|_{L^{\infty}(B_{5 / 4})} \leq C .$ Suppose the contrary, i.e., there exists a sequence of solutions $u_{i}$ of \eqref{eq:Integral} satisfying $\|u_{i}\|_{L^{\infty}(B_{5 / 4})} \rightarrow \infty$ as $i \rightarrow \infty$. For any fixed $\varepsilon>0$ sufficiently small and $R\gg1$, by Proposition \ref{pro:Proof2} the set $S_{i}$ associated to $u_{i}$ defined by Proposition \ref{pro:Proof1} only consists of finite many points in $B_{3 / 2}$ with a uniform positive lower bound of distances between each two points, if $S_{i} \cap B_{3 / 2}$ has points more than 1. By the contradiction assumption $\|u_{i}\|_{L^{\infty}(B_{5 / 4})} \rightarrow \infty$ and Proposition \ref{pro:Proof1}, $S_{i} \cap B_{11 / 8}$ is not empty and has only isolated blow up points of $\{u_{i}\}$ after passing to a subsequence. By Proposition \ref{pro:Estimates2}, these isolated blow up points have to be isolated simple blow up points. Suppose that $x_{i} \rightarrow \bar{x} \in \overline{B}_{11 / 8}$ is an isolated simple blow up point of $\{u_{i}\}$. By Proposition \ref{pro:Blowup6}, we have
$$
|u_{i}(x_{i})^{2} \mathcal{P}_{\sigma}(x_{i}, r, u_{i}, p_i)| \leq C(r).
$$
On the other hand, by the assumption of $a$ and Proposition \ref{pro:Estimates1} we have
$$
\liminf _{i \rightarrow \infty} u_{i}(x_{i})^{2} \mathcal{P}_{\sigma}(x_{i}, r, u_{i}, p_i)=\infty \quad\text { for some small }\, r>0
$$
if $n \geq 4 \sigma .$ Hence, we obtain a contraction and thus $\|u\|_{L^{\infty}(B_{5 / 4})} \leq C .$ The theorem then follows from interior estimates of solutions of linear integral equations in \cite{JLXThe2017}.
\end{proof}

{\noindent\bf Proof of Theorem \ref{thm:1.3}.}
\begin{proof}
For any fixed $\varepsilon>0$ sufficiently small and $R\gg1$, let $S_{i}$ be the set associated to $u_{i}$ defined by Proposition \ref{pro:Proof1}.

If $4 \sigma+2 \leq n<4 \sigma+4$, by Proposition \ref{pro:Proof2} the set $S_{i}$ only consists of finite many points in $B_{3 / 2}$. Since $u_{i}(x_{i}) \rightarrow \infty$ and $x_{i} \rightarrow \bar{x}$, by item (iii) of Proposition \ref{pro:Proof1}, after passing to subsequence, there exists $S_{i} \ni x_{i}^{\prime} \rightarrow \bar{x}$ is an isolated blow up point of $\{u_{i}\} .$ By Proposition \ref{pro:Estimates2}, it has to be an isolated simple blow up point. By Proposition \ref{pro:Blowup6}, we have
$$
|u_{i}(x_{i}^{\prime})^{2} \mathcal{P}_{\sigma}(x_{i}^{\prime}, r, u_{i}, p_i)| \leq C(r).
$$
By Proposition \ref{pro:Estimates1}, we proved the theorem for $4 \sigma+2 \leq n<4 \sigma+4$.

If $n \geq 4 \sigma+4$, suppose the contrary that, for some subsequence which we still denote as $i$,
\begin{equation}\label{64}
\begin{aligned}
&\sigma a_{i}(x_{i}) u_{i}(x_{i})^{\frac{4}{n-2 \sigma}} +\frac{\sigma+1}{2 n} \Delta a_{i}(x_{i}) \\
 \geq &\frac{1}{|o(1)|}\left\{\begin{array}{ll}
u_{i}(x_{i})^{-\frac{4}{n-2 \sigma}} \ln u_{i}(x_{i}) & \text { for }\, n=4 \sigma+4, \\
u_{i}(x_{i})^{-\frac{4}{n-2 \sigma}} & \text { for }\, n>4 \sigma+4.
\end{array}\right.
\end{aligned}
\end{equation}
Let $\mu_{i}=\operatorname{dist}\{x_{i}, S_{i} \backslash\{x_{i}\}\}$ and
$$
\Phi_{i}(x)=\mu_{i}^{(n-2 \sigma)/{2}} u_{i}(\mu_{i} x+x_{i}).
$$
If $x_{i} \notin S_{i}$, we have $u_{i}(x_{i}) \leq C \mu_{i}^{-(n-2 \sigma)/{2}}$. Hence, $\Phi_{i}(0) \leq C<\infty$ and $\mu_{i} \rightarrow 0 .$ Since $\max _{B_{\bar{d}}(x_{i})} u_{i}(x) \leq \bar{b} u_{i}(x_{i})$, we conclude that $\Phi_{i}(x) \leq C \bar{b}$ for all
$|x| \leq \bar{d} / \mu_{i} .$ By the argument of proof of Proposition \ref{pro:Blowup1}, for some $x_{0} \in \mathbb{R}^{n}$ and $\lambda>0$,
$$
\Phi_{i}(x) \rightarrow\Big(\frac{\lambda}{1+\bar{c} \lambda^{2}|x-x_{0}|^{2}}\Big)^{\frac{n-2 \sigma}{2}} \quad \text { in }\, C_{l o c}^{2}(\mathbb{R}^{n}).
$$
Note that the limiting function has only one critical point. Suppose $z_{i} \in S_{i}$ satisfying $|z_{i}-x_{i}|=\mu_{i} .$ Since $x_{i}$ and $z_{i}$ both are local maximum points of $\{u_{i}\}$, $\nabla \Phi_{i}(0)=0$ and, after passing to subsequence,
$$
\frac{z_{i}-x_{i}}{\mu_{i}} \rightarrow \bar{x} \text { with }\, |\bar{x}|=1, \quad 0=\nabla \Phi_{i}(\frac{z_{i}-x_{i}}{\mu_{i}}).
$$
We obtain a contradiction. Hence, $x_{i} \in S_{i}$. Therefore 0 is an isolated blow up point of $\{\Phi_{i}(x)\}$. By the assumptions of $a_i$, the contradiction assumption \eqref{64} and Proposition \ref{pro:Estimates2}, 0 is an isolated simple blow up point. Making use of Proposition \ref{pro:Blowup6} and Proposition \ref{pro:Estimates1} we obtain a contradiction again. Therefore, we complete the proof.
\end{proof}

{\noindent\bf Acknowledgments.}\, We would like to express our deep thanks to Professor Jingang Xiong for his encouragement on the subject of this paper.

\bigskip

\noindent M. Niu

\noindent College of Mathematics, Faculty of Science, Beijing University of Technology, Beijing, 100124, China\\[1mm]
Email: \textsf{niumiaomiao@bjut.edu.cn}\\

\noindent Z. Tang \& N. Zhou

\noindent School of Mathematical Sciences, Laboratory of Mathematics and Complex Systems, MOE, \\
Beijing Normal University, Beijing, 100875, China\\[1mm]
Email: \textsf{tangzw@bnu.edu.cn}\\
Email: \textsf{nzhou@mail.bnu.edu.cn}

\end{document}